\documentclass[a4paper,11pt]{article}

\usepackage[intlimits]{amsmath}
\usepackage{amssymb}
\usepackage{amsthm}
\usepackage{cite}
\usepackage{array}
\usepackage[T1]{fontenc}
\usepackage{lmodern}
\usepackage[a4paper,margin=30mm]{geometry}
\usepackage{microtype}
\usepackage{booktabs}
\usepackage[hidelinks,bookmarksnumbered]{hyperref}
\hypersetup{pdftitle={Uniqueness for the inverse Schrodinger problem in the plane with Lp potentials, p>1},
  pdfauthor={Catalin I. Carstea, Jenn-Nan Wang}}
\allowdisplaybreaks[1]

\newcommand{\pr}{\partial}

\newcommand{\Phiz}{\Phi_{z_0}}
\newcommand{\bPhiz}{\bar\Phi_{z_0}}
\newcommand{\psiz}{\psi_{z_0}}

\newcommand{\R}{\mathbb{R}}
\newcommand{\C}{\mathbb{C}}

\newcommand{\dd}{\,\text{d}}

\makeatletter
\newcommand{\aliaslabel}[1]{\ltx@label{#1}}
\makeatother

\theoremstyle{plain}
\newtheorem{thm}{Theorem}[section]
\newtheorem{prop}[thm]{Proposition}
\newtheorem{lem}[thm]{Lemma}
\newtheorem{cor}[thm]{Corollary}

\theoremstyle{definition}
\newtheorem{dfn}[thm]{Definition}

\theoremstyle{remark}
\newtheorem{rem}[thm]{Remark}

\title{Uniqueness for the inverse Schr\"odinger problem in the plane with $L^p$ potentials, $p>1$}
\author{C\u{a}t\u{a}lin I. C\^arstea\thanks{Department of Applied Mathematics, National Yang Ming Chiao Tung University, Hsinchu 30050, Taiwan. Email: \texttt{catalin.carstea@gmail.com}.}
\and Jenn-Nan Wang\thanks{Institute of Applied Mathematical Sciences, National Taiwan University, Taipei 106, Taiwan. Email: \texttt{jnwang@math.ntu.edu.tw}.}
}
\date{}

\begin{document}
\maketitle

\begin{abstract}
Let $\Omega\subset\mathbb R^2$ be a bounded smooth domain. We prove that the weak Dirichlet-to-Neumann map for $-\Delta+V$ uniquely determines every complex-valued potential $V\in L^p(\Omega)$, $p>1$, provided that zero is not a Dirichlet eigenvalue. This extends the previously known range $p>4/3$ to all $p>1$. The proof uses Bukhgeim's quadratic-phase solutions and an average of Alessandrini's identity over the phase center. After separating the Neumann-series tails, we show that each remaining Born term tends to zero. The fixed-order estimates combine bounds for the absolute kernels with oscillatory cancellation and a duality argument for the product of the two Cauchy transforms.
\end{abstract}



\section{Introduction}

\subsection{Setting and main result}

Let $\Omega\subset\R^2\equiv\C$ be a bounded smooth domain, and consider
\begin{equation}\label{eq}
\left\{\begin{array}{l}
-\triangle u+Vu=0,\quad \text{in }\Omega,\\[5pt]
 u|_{\partial\Omega}=f.
\end{array}\right.
\end{equation}
We ask whether the weak Dirichlet-to-Neumann map for \eqref{eq} determines the potential in the range
\begin{equation}\label{Lp-space}
        V\in L^p(\Omega),\qquad p>1.
\end{equation}
Our main result gives an affirmative answer, also for complex-valued potentials.

We use the following weak formulation.  Let $V\in L^p(\Omega)$, $p>1$, and write $p'=p/(p-1)$.  Since $H^1(\Omega)\hookrightarrow L^{2p'}(\Omega)$ in two dimensions, the product $w\widetilde w$ belongs to $L^{p'}(\Omega)$ for $w,\widetilde w\in H^1(\Omega)$.  Hence the bilinear form
\begin{equation}\label{weak-form-def}
        a_V(w,\widetilde w)=\int_\Omega \nabla w\cdot\nabla \widetilde w\,\dd x+
        \int_\Omega Vw\widetilde w\,\dd x
\end{equation}
is bounded on $H^1(\Omega)\times H^1(\Omega)$.  Our convention is complex bilinear: no conjugates occur either in this form or in the boundary pairings below.

Assume now that $0$ is not a Dirichlet eigenvalue of $-\Delta+V$.  For each $f\in H^{1/2}(\partial\Omega)$ there is a unique weak solution $u_f\in H^1(\Omega)$ with trace $f$.  Indeed, multiplication by $V$ is compact from $H^1_0(\Omega)$ to $H^{-1}(\Omega)$, because $H^1_0(\Omega)$ embeds compactly into $L^{2p'}(\Omega)$; hence $-\Delta+V:H^1_0(\Omega)\to H^{-1}(\Omega)$ is Fredholm of index zero, and injectivity gives solvability after subtracting an $H^1$ extension of $f$.  The weak Dirichlet-to-Neumann map is defined by the boundary pairing
\begin{equation}\label{DN-space}
        \Lambda_V:H^{1/2}(\partial\Omega)\longrightarrow H^{-1/2}(\partial\Omega),
        \qquad
        \langle \Lambda_V f,g\rangle=a_V(u_f,w_g),
\end{equation}
where $w_g\in H^1(\Omega)$ is any extension of $g\in H^{1/2}(\partial\Omega)$.  This is independent of the extension because $a_V(u_f,\cdot)$ vanishes on $H^1_0(\Omega)$.  The bilinear symmetry of $a_V$ also gives
\begin{equation*}
        \langle \Lambda_V f,g\rangle=\langle \Lambda_V g,f\rangle,
        \qquad f,g\in H^{1/2}(\partial\Omega).
\end{equation*}

We can now state the main theorem.
\begin{thm}\label{main-thm}
Let $1<p<\infty$, let $V,\widetilde V\in L^p(\Omega)$ be possibly complex-valued, and assume that $0$ is not a Dirichlet eigenvalue for either $-\Delta+V$ or $-\Delta+\widetilde V$ in $\Omega$.  If
\begin{equation*}
        \Lambda_V=\Lambda_{\widetilde V},
\end{equation*}
then $V=\widetilde V$ almost everywhere in $\Omega$.
\end{thm}

\subsection{Background and related work}

The inverse boundary value problem for the Schr\"odinger equation is closely related to Calder\'on's problem of determining an electrical conductivity from boundary voltage and current measurements \cite{Calderon1980}. In dimensions $n\geq3$, Sylvester and Uhlmann proved global uniqueness using complex geometrical optics (CGO) solutions with linear phases \cite{SylvesterUhlmann1987}. For unbounded potentials, Dos Santos Ferreira, Kenig, and Salo obtained uniqueness in the scale-invariant class $L^{n/2}$ on admissible geometries, including smooth bounded Euclidean domains \cite{DosSantosFerreiraKenigSalo2013}. This exponent also occurs in the unique-continuation theorem of Jerison and Kenig \cite{JerisonKenig1985}. In the plane, Amrein, Berthier, and Georgescu established unique continuation for potentials in $L^p_{\mathrm{loc}}$, $p>1$, in the corresponding local Sobolev solution classes \cite{AmreinBerthierGeorgescu1981}.

Early two-dimensional inverse results included local uniqueness under perturbative assumptions, due to Sylvester--Uhlmann and Sun \cite{SylvesterUhlmann1986,Sun1989,Sun1990}, and generic uniqueness, due to Sun and Uhlmann \cite{SunUhlmann1991}. Another line of work recovered singularities of the potential without determining the entire potential. Sun and Uhlmann studied this question for planar inverse scattering and related formally determined problems \cite{SunUhlmann1993Scattering,SunUhlmann1993Singularities}. Serov and P\"aiv\"arinta subsequently obtained estimates for the Green--Faddeev function and used the fixed-energy Born approximation to recover singularities \cite{SerovPaivarinta2005}. These results are relevant to low-regularity uniqueness because they separate a singular part determined by the data from a more regular remainder.

A parallel low-regularity theory has developed for the planar conductivity problem. Nachman established global uniqueness and reconstruction for uniformly positive conductivities in $W^{2,p}$, $p>1$, by a $\bar\partial$-scattering method \cite{Nachman1996}. Brown and Uhlmann proved uniqueness for uniformly positive conductivities in $W^{1,p}$, $p>2$ \cite{BrownUhlmann1997}, and Astala and P\"aiv\"arinta treated bounded measurable conductivities bounded away from zero \cite{AstalaPaivarinta2006}. For conductivities in $W^{1,2}$ that are bounded away from zero but may be unbounded above, C\^arstea and Wang proved uniqueness on simply connected domains under a smallness condition on $\|\nabla\log\gamma\|_{L^2}$ \cite{CarsteaWang2018}. Nachman, Regev, and Tataru obtained global uniqueness for positive conductivities with $\log\gamma\in\dot H^1$, using a nonlinear Plancherel theorem \cite{NachmanRegevTataru2020}. The conductivity structure is important in these results; it is not assumed for the complex potentials considered here.

There is also an earlier uniqueness theorem for real Schr\"odinger potentials throughout the range $p>1$ under a spectral positivity assumption. Isakov and Nachman proved that, if $V,\widetilde V\in L^p(\Omega)$ are real-valued, the lowest Dirichlet eigenvalue of $-\Delta+V$ is strictly positive, and zero is not a Dirichlet eigenvalue of $-\Delta+\widetilde V$, then equality of the Dirichlet-to-Neumann maps implies $V=\widetilde V$ \cite[Theorem~1.3]{IsakovNachman1995}. Their proof also gives a reconstruction procedure under suitable boundary regularity. The positivity assumption is stronger than the exclusion of zero from the Dirichlet spectrum used in Theorem~\ref{main-thm}.

For general planar potentials, Bukhgeim introduced CGO solutions with quadratic complex phases and obtained global uniqueness for smooth potentials \cite{Bukhgeim2008}. The stationary point of the phase allows the leading term in the boundary integral identity to recover the potential at an interior point. Bl\r{a}sten's account \cite{Blasten2011} discusses the regularity issues in this argument. Imanuvilov and Yamamoto proved uniqueness for $L^p$ potentials with $p>2$ \cite{ImanuvilovYamamotoLinear2012}; Bl\r{a}sten, Imanuvilov, and Yamamoto also treated this class and established conditional logarithmic stability estimates under additional Sobolev assumptions \cite{BlastenImanuvilovYamamoto2015}. Related developments for partial boundary data include the work of Imanuvilov, Uhlmann, and Yamamoto and the subsequent regularity refinement of Imanuvilov and Yamamoto \cite{ImanuvilovUhlmannYamamoto2010,ImanuvilovYamamoto2012}.

Pointwise reconstruction by quadratic-phase solutions has also been studied separately from uniqueness. Astala, Faraco, and Rogers obtained almost-everywhere reconstruction for compactly supported $H^{1/2}$ potentials and exhibited failures of the unaveraged reconstruction procedure below this regularity \cite{AstalaFaracoRogers2016}. Tejero introduced averaged reconstruction formulas for compactly supported complex potentials in $H^s(\mathbb R^2)$ for every $s>0$ \cite{Tejero2019}. In particular, averaging the reconstruction formula in the phase center permits recovery at lower Sobolev regularity. These results concern convergence of reconstruction formulas under positive Sobolev regularity; the present proof uses a fixed smooth test function and distributional convergence for potentials assumed only to belong to $L^p$.

The closest predecessor to Theorem~\ref{main-thm} is the global uniqueness theorem of Bl\r{a}sten, Tzou, and Wang for $L^p$ potentials with $p>4/3$ \cite{BlastenTzouWang2020}. Their construction of quadratic-phase CGO solutions already applies for every $p>1$. The additional restriction in their uniqueness argument enters in estimating the terms obtained after substituting these solutions into the integral identity. They use refined estimates at different Born orders together with an improvement in the integrability of the potential difference under equality of the boundary maps, obtained from the singularity-recovery result of Serov and P\"aiv\"arinta \cite{SerovPaivarinta2005}. Thus the construction of the solutions and the estimates for the integral identity impose different regularity requirements.

A different $\bar\partial$-scattering approach was developed by Lakshtanov and Vainberg for compactly supported $L^p$ potentials, $p>1$ \cite{LakshtanovVainberg2017}. Their reconstruction theorem is proved for real-valued potentials. For each prescribed point $z_0$, it gives reconstruction in a neighbourhood of $z_0$ for a generic set of potentials; this set may depend on $z_0$, and the neighbourhood may depend on the potential \cite[Theorem~2.3 and the following remark]{LakshtanovVainberg2017}. The authors also explain that their treatment of complex-valued potentials requires additional smoothness. These qualifications distinguish that reconstruction result from global uniqueness for arbitrary complex $L^p$ potentials.

Theorem~\ref{main-thm} extends the Bl\r{a}sten--Tzou--Wang uniqueness result to the remaining range $1<p\leq4/3$, with no reality, spectral positivity, or genericity assumption. The proof retains the quadratic-phase construction and the Neumann-series estimates. The additional work concerns the finitely many Born terms left after truncation: we average their integral identities over the phase center and prove convergence at the original $L^p$ regularity, without first improving the integrability of $V-\widetilde V$.

While completing this manuscript, we became aware of independent work by Ali Feizmohammadi, which proves the same theorem by substantially different methods. The two works were developed independently, and the respective authors agreed to post their manuscripts simultaneously on arXiv.

\subsection{Strategy of the proof}

It suffices to treat $1<p<2$, since on a bounded domain every $L^p$ potential with $p\geq2$ belongs to $L^{p_0}$ for any fixed $p_0\in(1,2)$.

The argument has three ingredients.

\paragraph{Averaging over the phase center.}
We use Bukhgeim's solutions with quadratic phase $\tfrac12(z-z_0)^2$ for $V$ and the conjugate phase for $\widetilde V$. They are normalized in the usual way, by subtracting from the Cauchy transforms $A=\bar\partial^{-1}(V/4)$ and $\widetilde A=\partial^{-1}(\widetilde V/4)$ their values at the center $z_0$. Instead of letting $\tau\to\infty$ at a fixed center, we multiply Alessandrini's identity by $\tau\varphi(z_0)/\pi$, with $\varphi\in C_0^\infty(\Omega)$, and integrate in $z_0$. The leading term becomes
\[
 \int_\Omega (V-\widetilde V)(z)E_\tau\varphi(z)\,\dd^2z,
\]
where $E_\tau$ is a unimodular Fourier multiplier with $E_\tau\varphi\to\varphi$ uniformly. Thus only $V-\widetilde V\in L^1$ is needed to recover the potential difference distributionally. Center averaging also appears in the $p>4/3$ argument of Bl\r{a}sten--Tzou--Wang \cite{BlastenTzouWang2020}; here we carry it through every fixed Born order.

\paragraph{Finitely many Born terms.}
After the Neumann series of the two solutions are expanded, the usual operator estimates show that all sufficiently high Born orders contribute $o(1)$. Each of the finitely many remaining terms is a nested integral of Cauchy kernels. In such a term the $z_0$-dependence is a sum of quadratic phases whose signs add up to one. Completing the square turns the $z_0$-integral into $E_\tau$ applied to one of the four functions
\[
 \varphi,\qquad \varphi A,\qquad \varphi\widetilde A,
 \qquad \varphi A\widetilde A,
\]
evaluated at an affine combination $c$ of the integration variables, times a residual quadratic oscillation independent of $z_0$. Replacing $E_\tau$ by the identity leaves a $\tau$-independent amplitude. Once this amplitude is shown to be integrable, the residual oscillation makes its contribution tend to zero. The errors made in this replacement are controlled by the uniform convergence $E_\tau\varphi\to\varphi$ and by the strong $L^2$ convergence of $E_\tau(\varphi A)$ and $E_\tau(\varphi\widetilde A)$. They are paired with positive densities obtained by integrating the absolute kernel over all variables except $c$. Bounding these densities is where $p>1$ enters: each inserted Born block acts boundedly because $|z|^{-1}*|z|^{-1}$ belongs locally to $L^{p'}$.

\paragraph{Duality for the product.}
The remaining error involves $E_\tau(\varphi A\widetilde A)-\varphi A\widetilde A$. Here $A\widetilde A\in L^{r/2}$, where $r=2p/(2-p)$, and $r/2<2$ when $p<4/3$, so $L^2$ convergence is unavailable. We instead transpose $E_\tau$ onto the oscillatory kernel. After transposition both chains of Cauchy transforms end at a smooth cutoff, and each terminal Cauchy transform gains $\tau^{-1/2}$. The two gains offset the prefactor $\tau$ and give a bound for the functional on $L^{r/2}$ that is uniform in $\tau$. Convergence then follows by approximating $\varphi A\widetilde A$ by smooth functions.

\subsection{Organization}

Sections~\ref{sec:preliminaries}--\ref{sec:cgo-construction} give the analytic estimates and the CGO construction. Section~\ref{sec:averaged-identity} reduces the averaged identity to finitely many Born terms. Section~\ref{sec:fixed-born-orders} explains the center selected by their phases and the common algebra of their terminal factors. Section~\ref{sec:positive-densities} proves the required density bounds and justifies the fixed-order identities at the original $L^p$ regularity. Section~\ref{sec:fixed-cancellations} proves cancellation, isolating the product error before estimating its transpose. Section~\ref{sec:final-proof} completes the proof.

An Isabelle/HOL formalization of the uniqueness argument, relative to explicitly identified results from the literature, together with translation and audit material and scripts for reproducing the checks, is available at \url{https://github.com/catalin-carstea/Lp-Schrodinger-Isabelle-formalization}.

\section{Notation and elementary tools}\label{sec:preliminaries}

\subsection{Standing notation}

We identify $\C$ with $\R^2$ and write $\dd^2z=\dd x\,\dd y$ for Lebesgue measure. Unless a domain is specified, integrals in complex variables are over $\C$. The notation $L^a_c(\C)$ denotes the compactly supported functions in $L^a(\C)$.

Until the final step we assume $1<p<2$.  Extend both potentials by zero outside $\Omega$, fix an open set $X$ with $\Omega\Subset X\Subset\C$, and write
\begin{equation}\label{Lp-q-notation}
        q=\frac V4,
        \qquad \widetilde q=\frac{\widetilde V}{4},
        \qquad Q=V-\widetilde V,
\end{equation}
\begin{equation*}
        A=\bar\partial^{-1}q,
        \qquad
        \widetilde A=\partial^{-1}\widetilde q.
\end{equation*}
Here $Q$ occurs in Alessandrini's identity, while $q$ and $\widetilde q$ generate the two Neumann series.  Their terminal factors involve the Cauchy transforms $A$ and $\widetilde A$.  Set
\begin{equation}\label{r-alpha-theta-sigma-def}
        r=\frac{2p}{2-p},
        \qquad
        p'=\frac{p}{p-1},
        \qquad
        \beta=\min\left\{1,\frac{r}{p'}\right\}.
\end{equation}
Since $q$ and $\widetilde q$ are compactly supported, the global Hardy--Littlewood--Sobolev estimate gives
\begin{equation}\label{global-cauchy-primitives}
        A,\widetilde A\in L^r(\C),
        \qquad
        \|A\|_{L^r(\C)}+\|\widetilde A\|_{L^r(\C)}
        \leq C_p\bigl(\|q\|_{L^p}+\|\widetilde q\|_{L^p}\bigr).
\end{equation}
They also belong to $W^{1,p}_{\mathrm{loc}}(\C)$ by the standard Cauchy-transform and Beurling-transform bounds recalled below.
Choose once and for all
\begin{equation}\label{section3-alpha-restated}
        0<\theta_0<1-\frac1p,
        \qquad
        \alpha=1-\frac1p-\theta_0>0,
\end{equation}
and set
\begin{equation}\label{sigma-def-general}
        \sigma=\frac{\beta}{2}+(1-\beta)\alpha.
\end{equation}
Here $r$ is the Sobolev exponent associated with $p$; $\alpha$ is the decay exponent used in the $L^\infty$ estimates, and $\sigma$ is the corresponding exponent in the $L^{p'}$ tail estimate. The parameter $\beta$ equals $r/p'$ in the interpolation case and $1$ when the fixed-support embedding $L^r\hookrightarrow L^{p'}$ applies.  We also write
\begin{equation}\label{tail-center-weights}
        \omega(z_0)=1+|A(z_0)|,
        \qquad
        \widetilde\omega(z_0)=1+|\widetilde A(z_0)|.
\end{equation}

The parameter $z_0$ is the center of the quadratic phase. We use the holomorphic phase $\Phi_{z_0}$ and the associated real phase $\psi_{z_0}=2\operatorname{Re}\Phi_{z_0}$:
\begin{equation}
\Phiz(z)=\frac{1}{2}(z-z_0)^2,
\qquad
\psiz(z)=\Phiz(z)+\bPhiz(z)=(x-x_0)^2-(y-y_0)^2.
\end{equation}
When the center is fixed in a local argument, we sometimes suppress it and write $\psi$ for $\psi_{z_0}$.  We use the complex derivatives
\begin{equation}
\pr=\frac{1}{2}\left(\pr_x-i\pr_y\right),
\qquad
\bar\pr=\frac{1}{2}\left(\pr_x+i\pr_y\right),
\end{equation}
so that
\begin{equation}
\pr\psiz=\pr\Phiz=z-z_0,
\qquad
\bar\pr\psiz=\bar\pr\bPhiz=\bar z-\bar z_0,
\qquad
\pr\bar\pr=\frac{1}{4}\triangle.
\end{equation}

\subsection{Cauchy transforms and localized kernels}

The inverse operators for these first-order derivatives are the Cauchy transforms.  For $f\in C_0^\infty(\C)$ we set
\begin{equation}\label{cauchy-transforms-def}
\pr^{-1} f(z)=\frac{1}{\pi}\int_{\C}\frac{f(z')}{\bar z-\bar z'}\dd^2z',\quad
\bar \pr^{-1} f(z)=\frac{1}{\pi}\int_{\C}\frac{f(z')}{ z-z'}\dd^2z'.
\end{equation}
These operators invert $\partial$ and $\bar\partial$ in the sense of distributions:
\begin{equation*}
        \pr\pr^{-1}f=f,\qquad \bar\pr\bar\pr^{-1}f=f.
\end{equation*}
We need only local Sobolev bounds and the Hardy--Littlewood--Sobolev gain for these transforms.  In particular, no global bound of the form $\partial^{-1}:L^a(\C)\to W^{1,a}(\C)$ is used for the Cauchy transform itself on the whole plane.

\begin{lem}\label{lem-cauchy-basic}
Let $X\Subset\C$ be bounded, and let $1<a<2$.  Define $a^*$ by
\begin{equation*}
        \frac1{a^*}=\frac1a-\frac12 .
\end{equation*}
If $f\in L^a(\C)$ is supported in a fixed bounded set, then
\begin{equation}\label{cauchy-HLS-bound}
        \|\partial^{-1}f\|_{L^{a^*}(X)}+
        \|\bar\partial^{-1}f\|_{L^{a^*}(X)}
        \leq C\|f\|_{L^a(\C)}.
\end{equation}
Moreover,
\begin{equation}\label{cauchy-local-W1a}\aliaslabel{cauchy-local-W1q}
        \|\partial^{-1}f\|_{W^{1,a}(X)}+
        \|\bar\partial^{-1}f\|_{W^{1,a}(X)}
        \leq C_X\|f\|_{L^a(\C)}.
\end{equation}
More generally, if $1<\kappa<\infty$ and $f\in L^\kappa(\C)$ is supported in a fixed bounded set $E$, then
\begin{equation}\label{cauchy-local-W1s-all}
        \|\partial^{-1}f\|_{W^{1,\kappa}(X)}+
        \|\bar\partial^{-1}f\|_{W^{1,\kappa}(X)}
        \leq C_{X,E,\kappa}\|f\|_{L^\kappa(\C)}.
\end{equation}
\end{lem}

\begin{proof}
These are standard consequences of the planar Hardy--Littlewood--Sobolev inequality and the $L^\kappa$-boundedness of the Beurling transform; see, for example, \cite{AstalaIwaniecMartin2009}.  The estimate \eqref{cauchy-HLS-bound} follows from the pointwise majorant by the Riesz kernel $|z|^{-1}$.  The identities $\partial\partial^{-1}f=f$ and $\bar\partial\bar\partial^{-1}f=f$, together with the Beurling-transform bounds for the cross derivatives, control the first derivatives.  For the local zero-order term, the case $1<\kappa<2$ follows from Hardy--Littlewood--Sobolev and boundedness of $X$, the case $\kappa=2$ from Young's inequality with the truncated kernel $\mathbf 1_{|z|\le R}|z|^{-1}\in L^1$, and the case $\kappa>2$ from the same inequality with that kernel in $L^{\kappa'}$, where $\kappa'<2$.
\end{proof}

For each fixed pair of
orders $(j,k)$ all variables in the associated finite expansions lie in a compact
set $X_{j,k}\Subset\C$, depending only on the fixed orders, supports, and cutoffs.  Thus every difference $z-w$ of two such variables
satisfies $|z-w|\le R_{j,k}$.  To simplify notation we suppress this harmless
dependence and write
\begin{equation}\label{localized-kernels-def}
        k_R(z)=\frac{\mathbf 1_{\{|z|\le R\}}}{|z|},
        \qquad
        K_R=k_R*k_R .
\end{equation}
The localized kernels satisfy
\begin{equation}\label{kernel-basic-facts-general}
        k_R\in L^\kappa(\C)\quad(1\le \kappa<2),
        \qquad
        K_R\in L^\kappa(\C)\quad(1\le \kappa<\infty).
\end{equation}
Indeed, the assertion for $k_R$ follows by polar coordinates.  For $K_R$, the case $\kappa=1$ follows from $k_R\in L^1$, while for $1<\kappa<\infty$ one may choose $a=2\kappa/(\kappa+1)<2$ and apply Young's inequality:
\begin{equation}\label{KR-Ls-young}
        \|K_R\|_{L^\kappa}
        \leq \|k_R\|_{L^a}^2,
        \qquad
        1+\frac1\kappa=\frac2a.
\end{equation}
We shall repeatedly use the following consequences.

\begin{lem}\label{lem-localized-kernel-mapping}
Let $1<p<2$ and $r=2p/(2-p)$.  On fixed compact supports,
\begin{equation}\label{kernel-mapping-facts-general}
        k_R*L^p\subset L^r,
        \qquad
        k_R*L^r\subset L^\infty,
        \qquad
        K_R*L^p\subset L^\infty .
\end{equation}
Moreover $k_R*L^2_c\subset L^\kappa_c$ for every finite $\kappa$.
\end{lem}

\begin{proof}
The first estimate in \eqref{kernel-mapping-facts-general} is the
Hardy--Littlewood--Sobolev estimate for the two-dimensional Riesz kernel
$|z|^{-1}$, since $1/r=1/p-1/2$, and the truncated kernel $k_R$ is pointwise
bounded by this kernel.  For the second estimate, $r>2$, hence
$r'=r/(r-1)<2$, and Young's inequality gives
\[
        \|k_R*f\|_\infty\le \|k_R\|_{r'}\|f\|_r .
\]
For the third estimate, \eqref{kernel-basic-facts-general} gives
$K_R\in L^{p'}$, and Young's inequality yields
\[
        \|K_R*f\|_\infty\le \|K_R\|_{p'}\|f\|_p .
\]
Finally, if $f\in L^2_c$ and $1\le \kappa\le2$, then $k_R\in L^1$ and Young's
inequality give $k_R*f\in L^2$, hence $k_R*f\in L^\kappa$ on fixed compact supports.
If $2<\kappa<\infty$, choose $\gamma<2$ so that $1+1/\kappa=1/2+1/\gamma$; then
$k_R\in L^\gamma$, and Young's inequality gives $k_R*f\in L^\kappa$.  All constants
depend only on $R$, the fixed compact supports, and the displayed exponents.
\end{proof}

\subsection{The center average and its transpose}

We use the Fourier-transform convention
\[
 \widehat f(\xi)=\int_{\R^2}e^{-ix\cdot\xi}f(x)\,\dd x,
 \qquad
 f(x)=\frac{1}{(2\pi)^2}\int_{\R^2}e^{ix\cdot\xi}\widehat f(\xi)\,\dd\xi.
\]
The leading quadratic oscillation is described by the following averaging operator.

\begin{lem}\label{lem-Etau}
For $\tau>0$ and $f\in\mathcal S(\C)$ define
\begin{equation}\label{Etau-def}
        E_\tau f(c)=\frac{\tau}{\pi}\int_{\C}e^{i\tau\psi_c(z_0)}f(z_0)\dd^2z_0,
        \qquad
        \psi_c(z_0)=(x_0-\operatorname{Re}c)^2-(y_0-\operatorname{Im}c)^2.
\end{equation}
With the Fourier normalization used here, $E_\tau$ is the multiplier operator
\begin{equation}\label{Etau-multiplier}
        \widehat{E_\tau f}(\xi)=m_\tau(\xi)\widehat f(\xi),
        \qquad
        m_\tau(\xi)=e^{-i\psi_0(\xi)/(4\tau)}.
\end{equation}
Consequently, $E_\tau$ extends uniquely to an isometry on $L^2(\C)$ and
\begin{equation}\label{Etau-strong}
        E_\tau f\longrightarrow f
        \qquad\text{strongly in }L^2(\C)
\end{equation}
for every $f\in L^2(\C)$.  Moreover, for every
$\varphi\in C_0^\infty(\C)$,
\begin{equation}\label{Etau-smooth-local}
        \|E_\tau\varphi-\varphi\|_{L^\infty(\C)}\longrightarrow0.
\end{equation}
For every $f\in L^1(\C)$, the integral in \eqref{Etau-def} is absolutely
convergent for every $c$ and defines a bounded function satisfying
\begin{equation}\label{Etau-L1-Linfty-fixed-tau}
        \|E_\tau f\|_{L^\infty(\C)}
        \leq \frac{\tau}{\pi}\|f\|_{L^1(\C)}.
\end{equation}
This $L^1$ realization agrees with the $L^2$ multiplier realization on
$L^1\cap L^2$.  In particular, if $F\in L^1(\C)$, then
\begin{equation}\label{Etau-smooth-pairing}
        \int_{\C}F(c)E_\tau\varphi(c)\dd^2c
        \longrightarrow
        \int_{\C}F(c)\varphi(c)\dd^2c.
\end{equation}
\end{lem}

\begin{proof}
The Fourier transform of $(\tau/\pi)e^{i\tau\psi_0}$ is the function
$m_\tau$ in \eqref{Etau-multiplier}.  Translation of the quadratic kernel
therefore gives \eqref{Etau-multiplier}.  Since $|m_\tau|=1$, Plancherel's
theorem shows that $E_\tau$ extends to an $L^2$ isometry.  Also
$m_\tau(\xi)\to1$ pointwise, and hence dominated convergence gives
\begin{equation*}
        \|(m_\tau-1)\widehat f\|_{L^2}\longrightarrow0
\end{equation*}
for every $f\in L^2(\C)$.  This proves \eqref{Etau-strong}.

If $\varphi\in C_0^\infty(\C)$, then $\widehat\varphi\in L^1(\C)$, and Fourier
inversion gives
\begin{equation*}
        \|E_\tau\varphi-\varphi\|_{L^\infty}
        \le C\int_{\C}|m_\tau(\xi)-1|\,|\widehat\varphi(\xi)|\dd^2\xi.
\end{equation*}
The right-hand side tends to zero by dominated convergence, proving
\eqref{Etau-smooth-local}.  The bound \eqref{Etau-L1-Linfty-fixed-tau} follows immediately from
\eqref{Etau-def}.  Agreement of the two realizations on $L^1\cap L^2$ follows
by approximating with Schwartz functions in both spaces.  The last assertion
then follows from
\begin{equation*}
        \left|\int F(E_\tau\varphi-\varphi)\right|
        \le \|F\|_{L^1}\|E_\tau\varphi-\varphi\|_{L^\infty}.
\end{equation*}
\end{proof}

\paragraph{The bilinear transpose.}
For compactly supported $f$ and $h$, define the transpose with respect to the bilinear pairing
$\int f(c)h(c)\,\dd^2c$ by
\begin{equation}\label{Etau-transpose-def}
        (E_\tau^t h)(z_0)
        =\frac{\tau}{\pi}\int e^{i\tau\psi_c(z_0)}h(c)\,\dd^2c
        =\frac{\tau}{\pi}\int e^{i\tau\psi_{z_0}(c)}h(c)\,\dd^2c .
\end{equation}
This is the transpose for the bilinear pairing, not the Hilbert-space adjoint; no complex conjugation is involved.  The equality of the two kernels follows from $\psi_c(z_0)=\psi_{z_0}(c)$.

\subsection{Quadratic phase algebra}

The elementary identity below separates the phase center from the remaining quadratic phase. For fixed spatial variables, the signed combination $c$ is the unique stationary point in $z_0$.

\begin{lem}\label{lem-signed-quadratic-normal-form}
Let $\varepsilon_0,\ldots,\varepsilon_N\in\{\pm1\}$ satisfy
$\sum_{\nu=0}^N\varepsilon_\nu=1$.  For spatial variables
$y_0,\ldots,y_N\in\C$ and center parameter $z_0\in\C$, put
$c=\sum_{\nu=0}^N\varepsilon_\nu y_\nu$.  Then
\begin{equation}\label{signed-quadratic-identity}
        \sum_{\nu=0}^N\varepsilon_\nu\psi_{z_0}(y_\nu)
        =\psi_{z_0}(c)+
        \operatorname{Re}\left(\sum_{\nu=0}^N\varepsilon_\nu y_\nu^2-c^2\right).
\end{equation}
The second term on the right is independent of $z_0$.  If $N\ge1$, its
restriction to each fiber $\{c=\mathrm{const}\}$, an affine subspace of real
dimension $2N$, has nondegenerate quadratic part.
\end{lem}

\begin{proof}
The identity follows by expanding
$\psi_{z_0}(z)=\operatorname{Re}((z-z_0)^2)$.  Since $c^2$ is constant on each
fiber, the quadratic part of the restriction is the real part of the complex
quadratic form $u\mapsto B(u,u)$ on
\begin{equation*}
        E=\{u\in\C^{N+1}:\textstyle\sum\varepsilon_\nu u_\nu=0\},
        \qquad
        B(u,v)=\sum\varepsilon_\nu u_\nu v_\nu .
\end{equation*}
If $u\in E$ annihilates $E$ under
$B$, then $u$ lies in the $B$-orthogonal complement of $E$, namely the line
spanned by $(1,\ldots,1)$.  This line intersects $E$ only at the origin, since
$\sum\varepsilon_\nu=1$.  Hence $u=0$.  To pass to the underlying real vector space, suppose that $u\in E$ and
$\operatorname{Re}B(u,v)=0$ for every $v\in E$.  Testing also with
$iv\in E$ gives $\operatorname{Im}B(u,v)=0$ for every $v\in E$, and therefore $u=0$ by
complex nondegeneracy.  Thus $\operatorname{Re}B$ is a
nondegenerate real quadratic form on $E$.
\end{proof}

\begin{lem}\label{lem-quadratic-RL}
Let $m\ge0$ and $d\ge1$, and let $P$ be a real quadratic polynomial on
$\R^m\times\R^d$, written in variables $(y,w)$.  Assume that the Hessian of
$w\mapsto P(y,w)$ is invertible.  Then, for every
$F\in L^1(\R^m\times\R^d)$,
\begin{equation}
        \int e^{i\tau P(y,w)}F(y,w)\,\dd y\,\dd w\longrightarrow0
        \qquad (\tau\to+\infty).
\end{equation}
\end{lem}

\begin{proof}
Since $P$ is quadratic, the Hessian $M$ of $w\mapsto P(y,w)$ is a constant
invertible symmetric matrix, and completing the square gives
\[
        P(y,w)=\tfrac12\bigl(w-w(y)\bigr)\cdot M\bigl(w-w(y)\bigr)+g(y)
\]
with $w(y)$ affine and $g$ real.  The Fourier transform of
$e^{i\tau w\cdot Mw/2}$ has modulus $C_M\tau^{-d/2}$.  Hence, by Parseval's
identity, for $F(y,w)=F_1(y)F_2(w)$ with $F_1\in L^1(\R^m)$ and
$F_2\in\mathcal S(\R^d)$,
\[
        \left|\int e^{i\tau P}F\right|
        \le \int|F_1(y)|
        \left|\int e^{i\tau w\cdot Mw/2}F_2(w+w(y))\,\dd w\right|\dd y
        \le C_M\tau^{-d/2}\|F_1\|_{L^1}\|\widehat{F_2}\|_{L^1},
\]
because translation does not change $|\widehat{F_2}|$.  Finite sums of such
products are dense in $L^1(\R^m\times\R^d)$, and
$|\int e^{i\tau P}F|\le\|F\|_{L^1}$ for every $\tau$.  The assertion follows.
\end{proof}

Lemma~\ref{lem-quadratic-RL} is applied to the residual phases of
Lemma~\ref{lem-signed-quadratic-normal-form} after a fixed affine change of
variables with constant Jacobian.  The passive variables $y$ consist of the
signed center $c$ (or a branch output) together with any integration variables
that do not occur in the phase, and $w$ are linear coordinates on the fiber
$\{c=\mathrm{const}\}$ of the variables that do occur.  The Hessian condition is
then the nondegeneracy statement of
Lemma~\ref{lem-signed-quadratic-normal-form}, and $d\ge1$ whenever the signed
list contains at least two variables.

\section{Oscillatory Cauchy estimates}\label{sec:oscillatory-cauchy}

We need two versions of the conjugated Cauchy-transform estimate used in Bukhgeim's construction; compare \cite{BlastenTzouWang2020}.  They follow from the same near--far decomposition and integration by parts, but have different roles.  The $L^\infty$ bound loses a logarithm, whereas the Hardy--Littlewood--Sobolev bound is log-free and later provides a uniform terminal gain of $\tau^{-1/2}$.

For a fixed center $z_0$ and parameter $\tau>0$, define
\begin{equation}
\pr_\psi^{-1}f=\pr^{-1}(e^{i\tau\psiz}f), \quad\bar\pr_\psi^{-1}f=\bar\pr^{-1}(e^{-i\tau\psiz}f),
\end{equation}
The dependence on $z_0$ and $\tau$ will usually be left implicit.

The first estimate maps $W^{1,b}$, $b>2$, into $L^\infty$.  It will provide the high-integrability endpoint in Lemma~\ref{lem-bukhgeim-interpolation}.

\subsection{The \texorpdfstring{$L^\infty$}{L-infinity} estimate}

\begin{lem}\label{lem-inf}
Let $2<b<\infty$ and let $Z\Subset\C$ be compact.  There is a constant $C=C_{X,Z,b}>0$ such that, uniformly for $z_0\in Z$ and $\tau\geq2$,
\begin{equation}
\|\pr_\psi^{-1}f\|_{L^\infty(X)}\leq  C\tau^{-\frac{1}{2}}\log(2+\tau) \|f\|_{W^{1,b}(X)},\quad \forall f\in W^{1,b}_0(X),
\end{equation}
where $f$ is viewed as extended by zero outside $X$.
\end{lem}
\begin{proof}
We regard $f$ as extended by zero outside $X$ and put
\begin{equation}
        R_0=1+\sup\{|z-z_0|:z\in\overline{X},\ z_0\in Z\}<\infty.
\end{equation}
Choose $\chi\in C_0^\infty(\C)$ with $0\leq \chi\leq1$, equal to one on $B(0,1)$ and to zero outside $B(0,2)$.  For $\delta>0$, let $\chi_\delta(z)=\chi(\delta^{-1}(z-z_0))$.

We split $\pr_\psi^{-1}f$ into
\begin{equation}
\pr_\psi^{-1} f=\pr_\psi^{-1}(\chi_\delta f)+\pr_\psi^{-1}\left((1-\chi_\delta)f\right).
\end{equation}
The near-center term has the kernel representation
\begin{equation}
\pr_\psi^{-1}(\chi_\delta f)(z)=\frac{1}{\pi}\int \frac{e^{i\tau\psiz(z')}}{\bar z-\bar z'}\chi(\delta^{-1}(z'-z_0))f(z')\dd^2 z',
\end{equation}
which is bounded by
\begin{equation}
\|\pr_\psi^{-1}(\chi_\delta f)\|_{L^\infty(X)}\leq C\|f\|_{L^\infty(X)}\int_0^{2\delta}\dd\rho=C\delta \|f\|_{W^{1,b}(X)}.
\end{equation}

For the term away from the center, integration by parts in the oscillation gives
\begin{multline}
\pr_\psi^{-1}\left((1-\chi_\delta)f\right)
=\frac{1}{i\tau}\pr^{-1}\left[(\pr e^{i\tau\psiz}) \frac{1-\chi_\delta}{\pr\psiz} f  \right]\\[5pt]
=\frac{1}{i\tau}e^{i\tau\psiz}\frac{1-\chi_\delta}{\pr\psiz} f
-\frac{1}{i\tau}\pr_\psi^{-1}\left[\pr\left(\frac{1-\chi_\delta}{\pr\psiz} f\right)\right].
\end{multline}
Since $\partial\psi_{z_0}=z-z_0$ and $\partial^2\psi_{z_0}=1$, the derivative in the last term is
\begin{equation}\label{product-rule-oscillatory-cauchy}
\partial\left(\frac{1-\chi_\delta}{\partial\psi_{z_0}}f\right)
=\frac{1-\chi_\delta}{\partial\psi_{z_0}}\partial f
-\frac{\partial\chi_\delta}{\partial\psi_{z_0}}f
-\frac{1-\chi_\delta}{(\partial\psi_{z_0})^2}f .
\end{equation}
We estimate the three terms in \eqref{product-rule-oscillatory-cauchy} separately.  The cutoff distance gives
\begin{equation}
\left\Vert \frac{1-\chi_\delta}{\pr\psiz} f\right\Vert_{L^\infty(X)}\leq\frac{\|f\|_{L^\infty(X)}}{\delta},
\end{equation}
and the local Cauchy estimate similarly gives
\begin{equation}
\left\Vert\pr_\psi^{-1}\left[\frac{1-\chi_\delta}{\pr\psiz} \pr f  \right]\right\Vert_{L^\infty(X)}\leq C\frac{\|\pr f\|_{L^b(X)}}{\delta}.
\end{equation}
The derivative of the cutoff is supported on the annulus $\delta\leq |z'-z_0|\leq2\delta$, where
\begin{multline}
\left|\pr_\psi^{-1}\left[\frac{\pr \chi_\delta}{\pr\psiz} f  \right](z)\right| \\
\leq C\delta^{-1}\|f\|_{L^\infty(X)}
\int_{\delta\leq|z'-z_0|\leq2\delta}
\frac{\dd^2z'}{|z-z'|\,|z'-z_0|}\\[5pt]
= C\delta^{-1}\|f\|_{L^\infty(X)}I(z-z_0),
\end{multline}
where
\begin{align}
I(z)
&=\int_{\delta\leq|z'|\leq2\delta}\frac{1}{|z-z'|}\frac{1}{|z'|}\dd^2z' \notag\\
&=\int_{\delta\leq|z'|\leq2\delta}\mathbf{1}_{B_{z}(\delta)}(z')\frac{1}{|z-z'|}\frac{1}{|z'|}\dd^2z' \notag\\
&\quad +\int_{\delta\leq|z'|\leq2\delta}\mathbf{1}_{B_{z}(\delta)^c}(z')\frac{1}{|z-z'|}\frac{1}{|z'|}\dd^2z'
=I_1(z)+I_2(z).
\end{align}
The portion of this annulus within distance $\delta$ of $z$ satisfies
\begin{equation}
I_1(z)\leq\int_{|z'-z|\leq \delta}\frac{1}{\delta|z'-z|}\dd^2 z'\leq C
\end{equation}
whereas its complement satisfies
\begin{equation}
I_2(z)\leq \int_{|z'|\leq2\delta}\frac{1}{\delta|z'|}\dd^2 z'\leq C.
\end{equation}
Therefore
\begin{equation}
\left|\pr_\psi^{-1}\left[\frac{\pr \chi_\delta}{\pr\psiz} f  \right](z)\right|\leq C\delta^{-1}\|f\|_{L^\infty(X)}.
\end{equation}

For the square-denominator term, the same decomposition gives
\begin{multline}
\left|\pr_\psi^{-1}\left[\frac{1- \chi_\delta}{(\pr\psiz)^2} f  \right](z)\right|\\[5pt]
\leq C\|f\|_{L^\infty(X)}\int_{R_0\geq|z'-z_0|\geq\delta}\frac{1}{|z-z'|}\frac{1}{|z'-z_0|^2}\dd^2z'\\[5pt]
=C\|f\|_{L^\infty(X)}J(z-z_0),
\end{multline}
where 
\begin{align}
J(z)
&=\int_{\delta\leq|z'|\leq R_0}\frac{1}{|z-z'|}\frac{1}{|z'|^{2}}\dd^2z' \notag\\
&=\int_{\delta\leq|z'|\leq R_0}\mathbf{1}_{B_{z}(\delta)}(z')\frac{1}{|z-z'|}\frac{1}{|z'|^{2}}\dd^2z' \notag\\
&\quad +\int_{\delta\leq|z'|\leq R_0}\mathbf{1}_{B_{z}(\delta)^c}(z')\frac{1}{|z-z'|}\frac{1}{|z'|^{2}}\dd^2z'
=J_1(z)+J_2(z).
\end{align}
Here
\begin{equation}
J_1(z)\leq\int_{|z'-z|\leq \delta}\frac{1}{\delta^2|z'-z|}\dd^2 z'\leq C\delta^{-1}
\end{equation}
and
\begin{equation}
J_2(z)\leq \int_{R_0\geq|z'|\geq\delta}\frac{1}{\delta|z'|^2}\dd^2 z'\leq C\delta^{-1}|\log\delta|.
\end{equation}

Combining these estimates,
\begin{equation}
\|\pr_\psi^{-1}f\|_{L^\infty(X)}\leq C\|f\|_{W^{1,b}(X)}\left(\delta+\tau^{-1}\delta^{-1} |\log\delta| \right).
\end{equation}
Choosing $\delta=\tau^{-1/2}$ gives the stated estimate.
\end{proof}

\subsection{The Hardy--Littlewood--Sobolev target estimate}

The second estimate has the Hardy--Littlewood--Sobolev target exponent and will give the $L^{r}$ bounds for the CGO remainders.
\begin{lem}\label{lem-qstar}
Let $1<a<2$, put $a^*=2a/(2-a)$, and let $Z\Subset\C$ be compact.  There is a constant $C=C_{X,Z,a}>0$ such that, uniformly for $z_0\in Z$ and $\tau\geq2$,
\begin{equation}
\|\pr_\psi^{-1}f\|_{L^{a^*}(X)}\leq C \tau^{-1/2}\|f\|_{W^{1,a}(X)}, \quad\forall f\in W^{1,a}_0(X).
\end{equation}
\end{lem}

\begin{proof}
Put
\begin{equation}
        R_0=1+\sup\{|z-z_0|:z\in\overline{X},\ z_0\in Z\}<\infty.
\end{equation}
With the notation of the preceding proof, the Hardy--Littlewood--Sobolev estimate gives
\begin{multline}
\|\partial_\psi^{-1}(\chi_\delta f)\|_{L^{a^*}(X)}\leq C\|\chi_\delta f\|_{L^a(X)}\\[5pt]
\leq C\|f\|_{L^{a^*}(X)}\|\chi_\delta\|_{L^2(\C)}= C\delta \|f\|_{L^{a^*}(X)}.
\end{multline}
Away from the center, the coefficient without a derivative obeys
\begin{equation}
\left\Vert \frac{1-\chi_\delta}{\pr\psiz} f\right\Vert_{L^{a^*}(X)}\leq\frac{\|f\|_{L^{a^*}(X)}}{\delta},
\end{equation}
and the term containing $\partial f$ satisfies
\begin{equation}
\left\Vert\pr_\psi^{-1}\left[\frac{1-\chi_\delta}{\pr\psiz} \pr f  \right]\right\Vert_{L^{a^*}(X)}\leq C\frac{\|\pr f\|_{L^a(X)}}{\delta}.
\end{equation}

For the cutoff-derivative term, supported on $\delta\le |z'-z_0|\le2\delta$,
\begin{multline}
\left|\pr_\psi^{-1}\left[\frac{\pr \chi_\delta}{\pr\psiz} f  \right](z)\right|
\leq C\delta^{-1}\int_{\delta\leq|z'-z_0|\leq2\delta}\frac{1}{|z-z'|}\frac{1}{|z'-z_0|}|f(z')|\dd^2z'\\[5pt]
=C\delta^{-1}\left( I_1(z-z_0)+I_2(z-z_0)\right),
\end{multline}
where
\begin{equation}
I_1(z)=\int_{\delta\leq|z'|\leq2\delta}\mathbf{1}_{B_{z}(\delta)}(z')\frac{1}{|z-z'|\,|z'|}|f(z'+z_0)|\dd^2 z'
\end{equation}
and
\begin{equation}
I_2(z)=\int_{\delta\leq|z'|\leq2\delta}\mathbf{1}_{B_{z}(\delta)^c}(z')\frac{1}{|z-z'|\,|z'|}|f(z'+z_0)|\dd^2 z'.
\end{equation}
Writing $1/\kappa=1-1/a^*$, we have
\begin{multline}
I_1(z)\leq\delta^{-1}\mathbf{1}_{B_0(5\delta)}(z)\int_{|z-z'|\leq\delta}\frac{1}{|z-z'|}|f(z'+z_0)|\dd^2 z'\\[5pt]
\leq \delta^{-1}\mathbf{1}_{B_0(5\delta)}(z)\|f\|_{L^{a^*}(X)}\left(\int_{|z-z'|\leq\delta}\frac{1}{|z-z'|^\kappa}\dd^2 z'\right)^{\frac{1}{\kappa}}\\[5pt]
\leq C \delta^{-\frac{2}{a^*}}\mathbf{1}_{B_0(5\delta)}(z)\|f\|_{L^{a^*}(X)},
\end{multline}
so
\begin{equation}
\|I_1\|_{L^{a^*}(X)}\leq C\|f\|_{L^{a^*}(X)}.
\end{equation}
The generalized Minkowski inequality gives
\begin{multline}
\|I_2\|_{L^{a^*}(X)}\leq C\int_{\delta\leq|z'|\leq2\delta}\left(\int_{\delta}^\infty \frac{1}{\rho^{a^*-1}}\dd\rho\right)^{\frac{1}{a^*}}\frac{1}{|z'|}|f(z'+z_0)|\dd^2 z'\\[5pt]
\leq C\delta^{\frac{2}{a^*}-1}\left(\int_\delta^{2\delta} \frac{1}{\rho^{\kappa-1}}\dd\rho\right)^{\frac{1}{\kappa}} \|f\|_{L^{a^*}(X)}
\leq C \|f\|_{L^{a^*}(X)}.
\end{multline}
It follows that
\begin{equation}
\left\Vert\pr_\psi^{-1}\left[\frac{\pr \chi_\delta}{\pr\psiz} f  \right]\right\Vert_{L^{a^*}(X)}\leq C\delta^{-1}\|f\|_{L^{a^*}(X)}.
\end{equation}

Split the square-denominator term in the same way:
\begin{multline}
\left|\pr_\psi^{-1}\left[\frac{1- \chi_\delta}{(\pr\psiz)^2} f  \right](z)\right| \\
\leq C
\int_{R_0\geq|z'-z_0|\geq\delta}
\frac{|f(z')|}{|z-z'|\,|z'-z_0|^2}\dd^2z'\\[5pt]
=C\left(J_1(z-z_0)+J_2(z-z_0)\right),
\end{multline}
where
\begin{equation}
J_1(z)=\int_{R_0\geq|z'|\geq\delta}\mathbf{1}_{B_z(\delta)}(z')\frac{1}{|z-z'|}\frac{1}{|z'|^2}|f(z'+z_0)|\dd^2z'
\end{equation}
and
\begin{equation}
J_2(z)=\int_{R_0\geq|z'|\geq\delta}\mathbf{1}_{B_z(\delta)^c}(z')\frac{1}{|z-z'|}\frac{1}{|z'|^2}|f(z'+z_0)|\dd^2z'.
\end{equation}
For $J_1$, introduce
\begin{equation}
        h_\delta(u)=\frac{\mathbf 1_{|u|\leq\delta}}{|u|},
        \qquad
        g_\delta(v)=\mathbf 1_{\delta\leq |v|\leq R_0}|v|^{-2}|f(v+z_0)| .
\end{equation}
Then $J_1\leq h_\delta*g_\delta$, and hence
\begin{equation}
        \|J_1\|_{L^{a^*}(X)}
        \leq \|h_\delta\|_{L^1}\|g_\delta\|_{L^{a^*}}
        \leq C\delta\cdot\delta^{-2}\|f\|_{L^{a^*}(X)}
        =C\delta^{-1}\|f\|_{L^{a^*}(X)} .
\end{equation}
By the generalized Minkowski inequality
\begin{multline}
\|J_2\|_{L^{a^*}(X)}\leq C\int_{\delta\leq|z'|\leq R_0}\left(\int_\delta^\infty\frac{1}{\rho^{a^*-1}}\dd\rho\right)^{\frac{1}{a^*}}\frac{1}{|z'|^2}|f(z'+z_0)|\dd^2 z'\\[5pt]
\leq C\delta^{\frac{2}{a^*}-1}\left(\int_\delta^{R_0}\frac{1}{\rho^{2\kappa-1}}\dd\rho\right)^{\frac{1}{\kappa}}\|f\|_{L^{a^*}(X)}\\[5pt]
\leq C\delta^{\frac{2}{a^*}-1}\delta^{\frac{2}{\kappa}-2}\|f\|_{L^{a^*}(X)}
=C\delta^{-1}\|f\|_{L^{a^*}(X)}.
\end{multline}

Altogether,
\begin{equation}
\|\pr_\psi^{-1}f\|_{L^{a^*}(X)}\leq C\|f\|_{W^{1,a}(X)}\left(\delta+\tau^{-1}\delta^{-1} \right).
\end{equation}
Choosing $\delta=\tau^{-1/2}$ gives the stated estimate.
\end{proof}

\begin{rem}\label{rem-cauchy-involution}
The estimates of Lemmas~\ref{lem-inf} and~\ref{lem-qstar} remain valid, with the same constants, after replacing $\partial$ by $\bar\partial$ and/or replacing $\psi$ by $-\psi$.  This follows either from the same proof or from conjugation together with reversal of the phase sign.  These changes will be used without further comment for the tilded CGO construction and for the finite Born branches.
\end{rem}

\subsection{Uniform terminal gains}

The product-error estimate will leave a smooth cutoff at the end of each Born branch. We record the corresponding Cauchy gain uniformly for centers in a fixed compact set. This is an application of Lemma~\ref{lem-qstar}, with the target exponent chosen for the later root pairing.
\begin{lem}\label{lem-terminal-cauchy-gain}
Let $2<\kappa<\infty$, and define
\begin{equation}\label{qs-terminal-def}
        \frac1{a_\kappa}=\frac12+\frac1\kappa,
        \qquad\text{equivalently}\qquad
        a_\kappa=\frac{2\kappa}{\kappa+2}.
\end{equation}
For $\varepsilon\in\{\pm1\}$ and
$\mathfrak d\in\{\partial,\bar\partial\}$, set
\begin{align}
\mathcal C_{\tau,z_0}^{\varepsilon,\partial}f(x)
 &=\frac1\pi\int
        \frac{e^{\varepsilon i\tau\psi_{z_0}(y)}f(y)}
             {\bar x-\bar y}\,\dd^2y,\label{terminal-cauchy-partial}\\
\mathcal C_{\tau,z_0}^{\varepsilon,\bar\partial}f(x)
 &=\frac1\pi\int
        \frac{e^{\varepsilon i\tau\psi_{z_0}(y)}f(y)}
             {x-y}\,\dd^2y .\label{terminal-cauchy-barpartial}
\end{align}
Let $X_0\Subset\C$ and $Z_0\Subset\C$ be fixed.  If
$f\in C_0^\infty(\C)$ is supported in a fixed compact set, then
\begin{equation}\label{terminal-cauchy-gain}
        \|\mathcal C_{\tau,z_0}^{\varepsilon,\mathfrak d}f\|_{L^\kappa(X_0)}
        \le C_{\kappa,X_0,Z_0}\tau^{-1/2}\|f\|_{W^{1,a_\kappa}(\C)}
\end{equation}
uniformly for $z_0\in Z_0$, $\tau\ge2$, $\varepsilon$, and $\mathfrak d$.  The constant
may depend on $X_0$, $Z_0$, and the fixed support of $f$.
\end{lem}

\begin{proof}
Choose a bounded open set $X$ containing $X_0$, $Z_0$, and the fixed support of
$f$, with $\operatorname{supp}f\Subset X$.  Since $a_\kappa^*=\kappa$, Lemma~\ref{lem-qstar}, applied with $a=a_\kappa$ and with the
center $z_0\in Z_0\subset X$, gives the asserted estimate on $X$, hence on
$X_0$, for the orientation $\partial$ and the sign $+$.  The other Cauchy
orientation and the opposite sign are covered by Remark~\ref{rem-cauchy-involution}.  The constants are uniform for $z_0\in Z_0$ because the
ambient set $X$ is fixed.
\end{proof}

\section{The CGO solutions}\label{sec:cgo-construction}

We now construct the two Bukhgeim solutions for $1<p<2$. The constants are uniform for $z_0\in\overline\Omega$ and large $\tau$; dependence on the values of $A$ and $\widetilde A$ at the center is displayed explicitly.

\subsection{Construction and interpolated operator bounds}

We begin with $V$ and obtain a Neumann expansion whose terms are controlled in both $L^\infty$ and $L^r$.  Recall that $q=V/4$, the normalization corresponding to $\Delta=4\partial\bar\partial$, and that $A=\bar\partial^{-1}q$.  By \eqref{global-cauchy-primitives} and Lemma~\ref{lem-cauchy-basic},
\begin{equation}
        A\in W^{1,p}_{\mathrm{loc}}(\C)\cap L^{r}(\C),
        \qquad r=\frac{2p}{2-p}.
\end{equation}
For a.e. $z_0\in\Omega$ the value $A(z_0)$ is defined as a Lebesgue value.  At such a center we subtract this value and put
\begin{equation}
        b_{z_0}(z)=A(z)-A(z_0).
\end{equation}
Then $b_{z_0}\in W^{1,p}_{\mathrm{loc}}(\C)\cap L^{r}_{\mathrm{loc}}(\C)$ and $\bar\partial b_{z_0}=q$.

Let $\chi\in C_0^\infty(X)$ satisfy $\chi=1$ in a neighbourhood of $\overline\Omega$.  The initial correction and the operator which inserts one more potential are
\begin{equation}\label{B-S-def-general}
        B_{z_0,\tau}=\partial_\psi^{-1}(\chi b_{z_0}),
        \qquad
        S^q_{z_0,\tau}f=\partial_\psi^{-1}\bigl(\chi\bar\partial_\psi^{-1}(qf)\bigr).
\end{equation}
If the Neumann equation
\begin{equation}
        W_{z_0,\tau}=B_{z_0,\tau}+S^q_{z_0,\tau}W_{z_0,\tau},
        \qquad
        \zeta_{z_0,\tau}=e^{-i\tau\psi_{z_0}}W_{z_0,\tau},
\end{equation}
is solvable in $L^\infty(X)$, then its solution satisfies
\begin{equation}\label{conjugated-remainder-equation}
        \bar\partial\left(e^{-i\tau\psi_{z_0}}\partial(e^{i\tau\psi_{z_0}}\zeta_{z_0,\tau})\right)
        =q(1+\zeta_{z_0,\tau})
\end{equation}
in $\Omega$, and consequently
\begin{equation}\label{cgo-u-def}
        u_{z_0,\tau}=e^{i\tau\Phi_{z_0}}(1+\zeta_{z_0,\tau})
\end{equation}
is a weak solution of $(-\Delta+V)u=0$ in $\Omega$.

\begin{lem}\label{lem-bukhgeim-interpolation}
Let $1<p<2$, let $q\in L^p(\Omega)$ be extended by zero to $X$, and define
\begin{equation}
        S^q_{z_0,\tau}f=\partial_\psi^{-1}\bigl(\chi\bar\partial_\psi^{-1}(qf)\bigr).
\end{equation}
For every $\theta_0>0$ there are constants $C_{\theta_0}$ and $\tau_0$, independent of $z_0\in\overline\Omega$ and $\tau$, such that for $\tau\geq \tau_0$,
\begin{equation}\label{S-Linf-interpolated}
        \|S^q_{z_0,\tau}f\|_{L^\infty(X)}
        \leq C_{\theta_0} \tau^{-(1-1/p)+\theta_0}
        \|q\|_{L^p(\Omega)}\|f\|_{L^\infty(X)} .
\end{equation}
Moreover,
\begin{equation}\label{S-Lpstar-from-Linf}
        \|S^q_{z_0,\tau}f\|_{L^{r}(X)}
        \leq C \tau^{-1/2}
        \|q\|_{L^p(\Omega)}\|f\|_{L^\infty(X)} .
\end{equation}
Finally, if $A=\bar\partial^{-1}q$ and $z_0$ is a Lebesgue point of $A$, then for
\begin{equation}
        B_{z_0,\tau}=\partial_\psi^{-1}\bigl(\chi(A-A(z_0))\bigr)
\end{equation}
one has
\begin{align}
        \|B_{z_0,\tau}\|_{L^{r}(X)}
        &\leq C \tau^{-1/2}
        \bigl(\|q\|_{L^p(\Omega)}+|A(z_0)|\bigr),       \label{B-Lpstar-interpolated}\\
        \|B_{z_0,\tau}\|_{L^\infty(X)}
        &\leq C_{\theta_0} \tau^{-(1-1/p)+\theta_0}
        \bigl(\|q\|_{L^p(\Omega)}+|A(z_0)|\bigr).       \label{B-Linf-interpolated}
\end{align}
\end{lem}

\begin{proof}
Apply Lemmas~\ref{lem-inf} and~\ref{lem-qstar} with $Z=\overline\Omega$.  To prove \eqref{S-Linf-interpolated}, first suppose that $q\in L^a(X)\cap L^b(X)$, where $1<a<p<b$ and $b>2$.  Lemmas~\ref{lem-inf} and~\ref{lem-cauchy-basic}, together with the $L^b$ boundedness of the Beurling transform, give
\begin{equation}\label{S-high-r}
        \|S^q_{z_0,\tau}f\|_{L^\infty(X)}
        \leq C_b \tau^{-1/2}\log(2+\tau)
        \|q\|_{L^b(\Omega)}\|f\|_{L^\infty(X)} .
\end{equation}
For $1<a<p$, the crude estimate has no decay but maps to the same target space.  Indeed, $a^*=2a/(2-a)>2$ because $a>1$, so local Sobolev embedding gives $W^{1,a^*}(X)\hookrightarrow L^\infty(X)$, and hence
\begin{align}
        \|S^q_{z_0,\tau}f\|_{L^\infty(X)}
        &\leq C_a\|S^q_{z_0,\tau}f\|_{W^{1,a^*}(X)}                                    \notag\\
        &\leq C_a
        \|\chi\bar\partial_\psi^{-1}(qf)\|_{L^{a^*}(X)}                                  \notag\\
        &\leq C_a
        \|\chi\bar\partial_\psi^{-1}(qf)\|_{W^{1,a}(X)}                                  \notag\\
        &\leq C_a
        \|q\|_{L^a(\Omega)}\|f\|_{L^\infty(X)} .                         \label{S-low-q}
\end{align}
For fixed $z_0,\tau$, the map $q\mapsto S^q_{z_0,\tau}$ is linear into
$\mathcal L(L^\infty(X))$, the Banach space of bounded operators on
$L^\infty(X)$.  Apply complex interpolation to the domain couple
$(L^a(X),L^b(X))$, keeping this target space fixed.  Interpolating
\eqref{S-low-q} and \eqref{S-high-r}, with
\(p^{-1}=(1-\theta)a^{-1}+\theta b^{-1}\), gives
\begin{equation}
        \|S^q_{z_0,\tau}\|_{L^\infty\to L^\infty}
        \leq C \tau^{-\theta/2}\log(2+\tau)^\theta
        \|q\|_{L^p(\Omega)} .
\end{equation}
Letting $a\downarrow1$ and $b\downarrow2$ makes $\theta/2$ arbitrarily close to $1-1/p$.  Absorbing the logarithm into an arbitrarily small loss in the power of $\tau$ gives \eqref{S-Linf-interpolated}.

It remains to identify this interpolation extension with the nested Cauchy formula.  For $q_1,q_2\in L^p(\Omega)$, the absolute kernels on the fixed supports give
\begin{equation}\label{S-coefficient-continuity}
 \bigl|S^{q_1-q_2}_{z_0,\tau}f(x)\bigr|
 \leq C K_R*\bigl(|q_1-q_2|\,|f|\bigr)(x),
\end{equation}
and hence, by $K_R\in L^{p'}$,
\begin{equation}
 \|S^{q_1-q_2}_{z_0,\tau}f\|_{L^\infty(X)}
 \leq C\|q_1-q_2\|_{L^p(\Omega)}\|f\|_{L^\infty(X)}.
\end{equation}
Thus approximation by coefficients in $L^a\cap L^b$ converges to the actual nested integral operator in $L^\infty(X)$, so the interpolation argument produces precisely $S^q_{z_0,\tau}$ as defined above.

The estimate \eqref{S-Lpstar-from-Linf} follows directly from Lemma \ref{lem-qstar} with exponent $p$:
\begin{align}
        \|S^q_{z_0,\tau}f\|_{L^{r}(X)}
        &\leq C \tau^{-1/2}
        \|\chi\bar\partial_\psi^{-1}(qf)\|_{W^{1,p}(X)}          \notag\\
        &\leq C \tau^{-1/2}
        \|q\|_{L^p(\Omega)}\|f\|_{L^\infty(X)} .
\end{align}
The $L^{r}$ estimate for $B_{z_0,\tau}$ is the same application of Lemma \ref{lem-qstar}, using
\begin{equation}
        \|\chi(A-A(z_0))\|_{W^{1,p}(X)}
        \leq C\bigl(\|q\|_{L^p(\Omega)}+|A(z_0)|\bigr),
\end{equation}
where the local $L^p$ norm of $A$ is controlled by Lemma~\ref{lem-cauchy-basic}.  For the $L^\infty$ estimate, apply the preceding interpolation argument to the linear map
\begin{equation}
        q\longmapsto \partial_\psi^{-1}\bigl(\chi\bar\partial^{-1}q\bigr).
\end{equation}
This gives the stated bound for the term with $A$.  The constant term $-A(z_0)\partial_\psi^{-1}\chi$ is estimated by Lemma \ref{lem-inf}; since $1-1/p<1/2$, the bound $\tau^{-1/2}\log(2+\tau)$ is stronger than the right hand side of \eqref{B-Linf-interpolated} after increasing $C_{\theta_0}$.
\end{proof}

\subsection{The Neumann series}

\begin{prop}\label{prop-cgo-precise}
Let $1<p<2$, let $V\in L^p(\Omega)$, put $q=V/4$, and let $A=\bar\partial^{-1}q$.  Fix
\begin{equation}\label{alpha-def-general}
        0<\theta_0<1-\frac1p,
        \qquad
        \alpha=1-\frac1p-\theta_0>0.
\end{equation}
For a.e. $z_0\in\Omega$ and all sufficiently large real $\tau$, the equation
\begin{equation}
        W=B_{z_0,\tau}+S^q_{z_0,\tau}W
\end{equation}
has a unique solution
\begin{equation}\label{W-neumann-series}
        W_{z_0,\tau}=\sum_{j=0}^\infty U_j,
        \qquad
        U_j=(S^q_{z_0,\tau})^jB_{z_0,\tau},
\end{equation}
in $L^\infty(X)$. The index $j$ counts the inserted two-Cauchy blocks; $U_0=B_{z_0,\tau}$ is already a correction term. Recall that $r=2p/(2-p)$. There is a constant $C_V$, depending on the fixed potential norm and on $\theta_0$, such that, with
\begin{equation}\label{rho-tau-def}
        \rho_\tau=C_V\tau^{-\alpha},
\end{equation}
one has $\rho_\tau\leq1/2$ for large $\tau$ and
\begin{align}
        \|U_j\|_{L^\infty(X)}
        &\leq C\bigl(1+|A(z_0)|\bigr)\tau^{-\alpha}\rho_\tau^j,
        \label{Uj-Linf-general}\\
        \|U_j\|_{L^r(X)}
        &\leq C\bigl(1+|A(z_0)|\bigr)\tau^{-1/2}\rho_\tau^j,
        \qquad j\geq0 .
        \label{Uj-Lr-general}
\end{align}
Consequently
\begin{equation}
        u_{z_0,\tau}=e^{i\tau\Phi_{z_0}}
        \left(1+e^{-i\tau\psi_{z_0}}W_{z_0,\tau}\right)
\end{equation}
is a weak solution of $(-\Delta+V)u=0$ in $\Omega$.
\end{prop}

\begin{proof}
By Lemma \ref{lem-bukhgeim-interpolation},
\begin{equation}
        \|S^q_{z_0,\tau}\|_{L^\infty\to L^\infty}
        \leq C_{\theta_0}\tau^{-\alpha}\|q\|_{L^p(\Omega)}.
\end{equation}
Choosing $\tau$ sufficiently large makes the right hand side at most $1/2$, and the Neumann series converges in $L^\infty(X)$.  The $L^\infty$ estimate \eqref{Uj-Linf-general} follows from \eqref{B-Linf-interpolated} and the contraction estimate.  The $L^r$ estimate for $j=0$ is \eqref{B-Lpstar-interpolated}.  For $j\geq1$, combine \eqref{S-Lpstar-from-Linf} with \eqref{Uj-Linf-general} applied to $U_{j-1}$; after increasing the constant depending on the fixed potential norm, this gives \eqref{Uj-Lr-general}.

The identity \eqref{conjugated-remainder-equation} follows from the Neumann equation.  Because $q$ was extended by zero and $\chi=1$ on an open neighbourhood $U$ of $\overline\Omega$, the same calculation gives the conjugated equation throughout $U$.  Since $W_{z_0,\tau}\in L^\infty(X)$, the function $u_{z_0,\tau}$ is locally bounded and satisfies
\begin{equation}
        \Delta u_{z_0,\tau}=Vu_{z_0,\tau}\in L^p(U).
\end{equation}
Local elliptic regularity gives $u_{z_0,\tau}\in W^{2,p}_{\mathrm{loc}}(U)$, and hence in $W^{2,p}$ on a smaller neighbourhood of $\overline\Omega$.  Since $p>1$ in dimension two, this implies $u_{z_0,\tau}\in H^1(\Omega)$ with the required trace.  Thus $u_{z_0,\tau}$ is an admissible weak solution.
\end{proof}

\subsection{The opposite Cauchy orientation}

For $\widetilde V$ we reverse the factorization and use the changes of orientation and phase sign from Remark~\ref{rem-cauchy-involution}.  Set
\begin{equation}
        \widetilde b_{z_0}(z)=\widetilde A(z)-\widetilde A(z_0),
\end{equation}
for $z_0$ in the full-measure set of Lebesgue points of $\widetilde A$.  The opposite oscillatory Cauchy inverses are
\begin{equation}
        \partial_{-\psi}^{-1}h=\partial^{-1}(e^{-i\tau\psi_{z_0}}h),
        \qquad
        \bar\partial_{-\psi}^{-1}h=\bar\partial^{-1}(e^{i\tau\psi_{z_0}}h),
\end{equation}
and the tilded initial correction and Neumann operator are
\begin{equation}\label{tilde-B-S-def-precise}
        \widetilde B_{z_0,\tau}=\bar\partial_{-\psi}^{-1}(\chi\widetilde b_{z_0}),
        \qquad
        \widetilde S_{z_0,\tau}f=\bar\partial_{-\psi}^{-1}\bigl(\chi\partial_{-\psi}^{-1}(\widetilde q f)\bigr).
\end{equation}
If
\begin{equation}
        \widetilde W=\widetilde B_{z_0,\tau}+\widetilde S_{z_0,\tau}\widetilde W,
        \qquad
        \widetilde\zeta_{z_0,\tau}=e^{-i\tau\psi_{z_0}}\widetilde W,
\end{equation}
then the conjugated equation is
\begin{equation}\label{tilde-conjugated-remainder-equation}
        \partial\left(e^{-i\tau\psi_{z_0}}\bar\partial(e^{i\tau\psi_{z_0}}\widetilde\zeta_{z_0,\tau})\right)
        =\widetilde q(1+\widetilde\zeta_{z_0,\tau})
\end{equation}
in $\Omega$, and
\begin{equation}
        \widetilde u_{z_0,\tau}=e^{i\tau\bar\Phi_{z_0}}(1+\widetilde\zeta_{z_0,\tau})
\end{equation}
solves $(-\Delta+\widetilde V)\widetilde u=0$ weakly in $\Omega$.  If
\begin{equation}
        \widetilde U_j=\widetilde S_{z_0,\tau}^j\widetilde B_{z_0,\tau},
        \qquad
        \widetilde\rho_\tau=C_{\widetilde V}\tau^{-\alpha},
\end{equation}
then the same argument proves \eqref{Uj-Linf-general}--\eqref{Uj-Lr-general}, with $U_j$, $A(z_0)$, and $\rho_\tau$ replaced by $\widetilde U_j$, $\widetilde A(z_0)$, and $\widetilde\rho_\tau$.

\subsection{Measurability and center domination}

\begin{lem}\label{lem-center-measurability}
Choose measurable representatives of $A$ and $\widetilde A$, and let
$\mathcal Z\subset\Omega$ be a common full-measure set of Lebesgue points of these
representatives.  For $z_0\notin\mathcal Z$, define the initial corrections, all finite
Neumann iterates, and the Neumann sums to be zero.  Fix $\tau$ sufficiently
large that
\begin{equation*}
        \rho_\tau\leq\frac12,
        \qquad
        \widetilde\rho_\tau\leq\frac12.
\end{equation*}
Then every finite iterate admits a jointly measurable representative on
$\Omega\times X$:
\begin{equation*}
        (z_0,z)\longmapsto U_j(z;z_0,\tau),
        \qquad
        (z_0,z)\longmapsto \widetilde U_j(z;z_0,\tau).
\end{equation*}
These representatives may be chosen so that, outside one set of product
measure zero, the estimates
\begin{align}
 |U_j(z;z_0,\tau)|
 &\leq C\bigl(1+|A(z_0)|\bigr)\tau^{-\alpha}\rho_\tau^j,
                                                        \label{Uj-pointwise-majorant}\\
 |\widetilde U_j(z;z_0,\tau)|
 &\leq C\bigl(1+|\widetilde A(z_0)|\bigr)
        \tau^{-\alpha}\widetilde\rho_\tau^j             \label{tilde-Uj-pointwise-majorant}
\end{align}
hold simultaneously for all $j\geq0$.  In particular, the two Neumann series
converge absolutely for almost every $(z_0,z)$ and define jointly measurable
representatives of the $L^\infty(X)$ Neumann sums.  They satisfy
\begin{align}
 |W_{z_0,\tau}(z)|
 &\leq C\bigl(1+|A(z_0)|\bigr)\tau^{-\alpha},
                                                        \label{W-pointwise-majorant}\\
 |\widetilde W_{z_0,\tau}(z)|
 &\leq C\bigl(1+|\widetilde A(z_0)|\bigr)\tau^{-\alpha}  \label{tilde-W-pointwise-majorant}
\end{align}
for almost every $(z_0,z)$.

For every $\varphi\in C_0^\infty(\Omega)$ one has
\begin{equation}\label{center-weight-integrability}
        \int_\Omega |\varphi(z_0)|
        \bigl(1+|A(z_0)|\bigr)
        \bigl(1+|\widetilde A(z_0)|\bigr)\,\dd^2z_0<\infty .
\end{equation}
Consequently, after multiplication by $\varphi(z_0)Q(z)$, the one-sided Neumann series and the mixed double series are absolutely integrable in $(z_0,z)$.  In particular, the center average and the Neumann sums may be interchanged term by term.
\end{lem}

\begin{proof}
\noindent\emph{Finite iterates.}
We use the following elementary fact.  The kernels of $\partial^{-1}$ and
$\bar\partial^{-1}$ are measurable and locally
integrable on $\C$.  If $h(z_0,z')$ is jointly measurable, compactly supported
in $z'$, and the corresponding Cauchy integral is finite almost everywhere,
then
\begin{equation*}
        (z_0,z)\longmapsto
        \int \frac{h(z_0,z')}{z-z'}\,\dd^2z'
\end{equation*}
is jointly measurable after an arbitrary definition on the exceptional null set.  Indeed, truncate the Cauchy kernel to $n^{-1}\leq |z-z'|\leq n$ and replace $h$ by
$h\mathbf 1_{\{|h|\leq n\}}$.  The parameterized integral theorem applies to
each bounded truncation, and dominated convergence gives the original integral
wherever the original integral is absolutely convergent.  The argument is unchanged for the conjugated kernel and after multiplication by the continuous factors $e^{\pm i\tau\psi_{z_0}(z')}$.

The maps
\begin{equation*}
        (z_0,z)\longmapsto A(z)-A(z_0),
        \qquad
        (z_0,z)\longmapsto \widetilde A(z)-\widetilde A(z_0)
\end{equation*}
are jointly measurable.  For almost every $z_0$, the compactly supported
functions $\chi(A-A(z_0))$ and
$\chi(\widetilde A-\widetilde A(z_0))$ belong to $L^r$, where $r>2$.
Since the localized Cauchy kernel belongs to $L^{r'}$ and $r'<2$, the defining
integrals for $B_{z_0,\tau}$ and $\widetilde B_{z_0,\tau}$ are absolutely
convergent for almost every $(z_0,z)$.  The preceding observation therefore
gives jointly measurable representatives of $U_0$ and $\widetilde U_0$.

Assume that $U_j$ is jointly measurable.  For almost every fixed
$z_0$, Proposition~\ref{prop-cgo-precise} gives $U_j\in L^\infty(X)$, and hence
$qU_j\in L^p$.  The first Cauchy transform is absolutely convergent almost
everywhere and belongs locally to $L^r$ by the Hardy--Littlewood--Sobolev
estimate.  After multiplication by $\chi$, the second Cauchy transform is
absolutely convergent because $r>2$.  Both transforms are jointly measurable by
the preceding parameterized-integral argument.  This proves the assertion for
$U_{j+1}$, and induction gives joint measurability of every finite iterate.  The
tilded induction is identical.

\par\smallskip\noindent\emph{Neumann sums.}
For each $j$, Proposition~\ref{prop-cgo-precise} gives, for almost every
$z_0\in\mathcal Z$,
\begin{equation*}
        \|U_j(\,\cdot\,;z_0,\tau)\|_{L^\infty(X)}
        \leq C\omega(z_0)\tau^{-\alpha}\rho_\tau^j.
\end{equation*}
Because the representative is jointly measurable, the set on which the
corresponding pointwise inequality fails is measurable.  Its $z$-section has
measure zero for almost every $z_0$, and hence Fubini's theorem shows that the
set itself has product measure zero.  Taking the countable union over $j$, and
then adjoining the analogous tilded exceptional sets, gives one product-null
set outside which \eqref{Uj-pointwise-majorant} and
\eqref{tilde-Uj-pointwise-majorant} hold for every $j$ simultaneously.

The two geometric series therefore converge absolutely outside that null set.
Their pointwise sums are jointly measurable.  For almost every fixed $z_0$,
they agree almost everywhere in $z$ with the $L^\infty(X)$ limits constructed
in Proposition~\ref{prop-cgo-precise}.  Since $\rho_\tau$ and
$\widetilde\rho_\tau$ are at most $1/2$, summing the pointwise bounds proves
\eqref{W-pointwise-majorant} and \eqref{tilde-W-pointwise-majorant}.

\par\smallskip\noindent\emph{Integrability in the center.}
On the compact support of $\varphi$, we have $A,\widetilde A\in L^r$, where
\begin{equation*}
        r=\frac{2p}{2-p}>2.
\end{equation*}
Hence $A$, $\widetilde A$, and $A\widetilde A$ are integrable there, the last because $A\widetilde A\in L^{r/2}$ and $r/2>1$.  This proves
\eqref{center-weight-integrability}.

Since $Q\in L^1(\Omega)$ and the oscillatory factors have modulus one,
\eqref{Uj-pointwise-majorant}--\eqref{tilde-W-pointwise-majorant} imply
\begin{align*}
 &\sum_{j\geq0}\int_\Omega\int_\Omega
 |\varphi(z_0)|\,|Q(z)|\,|U_j(z;z_0,\tau)|\,\dd^2z\,\dd^2z_0\\
 &\qquad\leq
 C\tau^{-\alpha}\|Q\|_{L^1(\Omega)}
 \int_\Omega |\varphi(z_0)|\omega(z_0)\,\dd^2z_0<\infty,
\end{align*}
and the same estimate holds for the tilded one-sided series.  Similarly,
\begin{align*}
 &\sum_{j,k\geq0}\int_\Omega\int_\Omega
 |\varphi(z_0)|\,|Q(z)|\,
 |U_j(z;z_0,\tau)|\,|\widetilde U_k(z;z_0,\tau)|
 \,\dd^2z\,\dd^2z_0\\
 &\qquad\leq
 C\tau^{-2\alpha}\|Q\|_{L^1(\Omega)}
 \int_\Omega |\varphi(z_0)|\omega(z_0)\widetilde\omega(z_0)
 \,\dd^2z_0<\infty.
\end{align*}
Tonelli's theorem therefore justifies the center average and the termwise expansion of the Neumann series.  Changes of variables inside a fixed Born kernel require a separate absolute-convergence argument, given in Lemma~\ref{lem-fixed-order-absolute-convergence} below.
\end{proof}

\section{The averaged identity and the Born decomposition}\label{sec:averaged-identity}\label{sec:reduction-cancellation}

Equality of the boundary maps first gives an interior identity.  Averaging this identity over the center of the quadratic phase identifies its leading term and reduces the remainder to finitely many Born orders.

\subsection{The weak Alessandrini identity}

\begin{lem}\label{lem-alessandrini}
Let $V,\widetilde V\in L^p(\Omega)$ for some $p>1$, and assume that the corresponding Dirichlet problems are well posed in the sense defined before Theorem \ref{main-thm}.  If $u,\widetilde u\in H^1(\Omega)$ solve
\begin{equation}
        (-\Delta+V)u=0,
        \qquad
        (-\Delta+\widetilde V)\widetilde u=0
\end{equation}
weakly in $\Omega$, then
\begin{equation}\label{weak-alessandrini}
        \int_\Omega (V-\widetilde V)u\widetilde u\,\dd x
        =\langle \Lambda_V u|_{\partial\Omega},\widetilde u|_{\partial\Omega}\rangle
        -\langle \Lambda_{\widetilde V}\widetilde u|_{\partial\Omega},u|_{\partial\Omega}\rangle .
\end{equation}
Consequently, if $\Lambda_V=\Lambda_{\widetilde V}$, then
\begin{equation}\label{weak-alessandrini-zero}
        \int_\Omega (V-\widetilde V)u\widetilde u\,\dd x=0.
\end{equation}
\end{lem}

\begin{proof}
Put $f=u|_{\partial\Omega}$ and $g=\widetilde u|_{\partial\Omega}$.  In the definition of $\Lambda_V f$ we may take $w_g=\widetilde u$, and in the definition of $\Lambda_{\widetilde V}g$ we may take the extension of $f$ to be $u$.  Thus
\begin{align*}
        \langle \Lambda_V f,g\rangle
        &=\int_\Omega \nabla u\cdot\nabla \widetilde u\,\dd x+
          \int_\Omega Vu\widetilde u\,\dd x,\\
        \langle \Lambda_{\widetilde V}g,f\rangle
        &=\int_\Omega \nabla \widetilde u\cdot\nabla u\,\dd x+
          \int_\Omega \widetilde V\widetilde u u\,\dd x.
\end{align*}
The gradient terms are equal in the bilinear convention used here.  Subtracting gives \eqref{weak-alessandrini}.  If $\Lambda_V=\Lambda_{\widetilde V}$, then
\begin{equation*}
        \langle \Lambda_V f,g\rangle
        =\langle \Lambda_{\widetilde V} f,g\rangle
        =\langle \Lambda_{\widetilde V} g,f\rangle,
\end{equation*}
where the last equality uses the bilinear symmetry of the weak Dirichlet-to-Neumann map for $\widetilde V$.  Hence the right side of \eqref{weak-alessandrini} vanishes.
\end{proof}

\subsection{Averaging in the phase center}

Instead of recovering the potential pointwise at a fixed center, we average against an arbitrary test function.  The leading quadratic phase then becomes $E_\tau\varphi$, so that only $Q\in L^1$ is needed for distributional recovery.

Applying Lemma~\ref{lem-alessandrini} to the CGO solution from Proposition~\ref{prop-cgo-precise} and to its tilded analogue gives, for a.e. $z_0$,
\begin{equation}\label{alessandrini}
\int Q(z) e^{i\tau\psiz(z)}(1+\zeta_\tau(z))(1+\widetilde\zeta_\tau(z))\dd^2 z=0.
\end{equation}
Expand the two remainders as
\begin{equation}\label{remainder-neumann-expansion}
\begin{aligned}
        \zeta_\tau&=e^{-i\tau\psi_{z_0}}W_{z_0,\tau},
        &\qquad
        \widetilde\zeta_\tau&=e^{-i\tau\psi_{z_0}}\widetilde W_{z_0,\tau},\\
        W_{z_0,\tau}&=\sum_{j=0}^\infty U_j,
        &\qquad
        \widetilde W_{z_0,\tau}&=\sum_{k=0}^\infty \widetilde U_k .
\end{aligned}
\end{equation}
Here $U_j=(S^q_{z_0,\tau})^jB_{z_0,\tau}$ and $\widetilde U_k=\widetilde S_{z_0,\tau}^k\widetilde B_{z_0,\tau}$.  Hence
\begin{equation}\label{pointwise-born-decomposition}
 e^{i\tau\psi_{z_0}}(1+\zeta_\tau)(1+\widetilde\zeta_\tau)
 = e^{i\tau\psi_{z_0}}
   +W_{z_0,\tau}+\widetilde W_{z_0,\tau}
   +e^{-i\tau\psi_{z_0}}W_{z_0,\tau}\widetilde W_{z_0,\tau}.
\end{equation}
The exterior phase cancels in the one-sided terms and remains in the mixed product.  Multiplying \eqref{alessandrini} by $\tau\varphi(z_0)/\pi$ and integrating in $z_0$ gives
\begin{equation}\label{averaged-born-decomposition}
0=\mathcal L_\tau
  +\sum_{j\ge0}\mathcal B_{j,\tau}
  +\sum_{k\ge0}\widetilde{\mathcal B}_{k,\tau}
  +\sum_{j,k\ge0}\mathcal M_{j,k,\tau},
\end{equation}
where
\begin{align}
\mathcal L_\tau
&=\frac{\tau}{\pi}\int\varphi(z_0)
  \int Q(z)e^{i\tau\psi_{z_0}(z)}\dd^2z\dd^2z_0, \label{leading-functional-def}\\
\mathcal B_{j,\tau}
&=\frac{\tau}{\pi}\int\varphi(z_0)
  \int Q(z)U_j(z;z_0,\tau)\dd^2z\dd^2z_0,
  \label{left-born-functional-def}\\
\widetilde{\mathcal B}_{k,\tau}
&=\frac{\tau}{\pi}\int\varphi(z_0)
  \int Q(z)\widetilde U_k(z;z_0,\tau)\dd^2z\dd^2z_0,
  \label{right-born-functional-def}\\
\mathcal M_{j,k,\tau}
&=\frac{\tau}{\pi}\int\varphi(z_0)
  \int Q(z)e^{-i\tau\psi_{z_0}(z)}
  U_j(z;z_0,\tau)\widetilde U_k(z;z_0,\tau)\dd^2z\dd^2z_0.
  \label{mixed-born-functional-def}
\end{align}

By Lemma~\ref{lem-center-measurability}, every term in this identity is measurable, the Neumann sums are absolutely integrable after multiplication by $\varphi(z_0)Q(z)$, and the displayed series may be integrated term by term.

\subsection{The leading term}

Let $\varphi\in C_0^\infty(\Omega)$. Since $\psi_{z_0}(z)=\psi_z(z_0)$, Fubini's theorem identifies the leading term with $\int Q E_\tau\varphi$. Here $Q\in L^1(\Omega)$, and Lemma~\ref{lem-Etau} gives $E_\tau\varphi\to\varphi$ uniformly on $\overline\Omega$. Thus
\begin{equation}\label{leading-limit}
 \mathcal L_\tau
 =\int Q(z)E_\tau\varphi(z)\dd^2z
 =\int Q(z)\varphi(z)\dd^2z+o(1).
\end{equation}

\subsection{The Neumann tails}

The high-order terms are controlled by the usual geometric summation of the Bukhgeim Neumann series, as in \cite{BlastenTzouWang2020}.

\begin{prop}\label{prop-neumann-tails}
Let $1<p<2$ and $\varphi\in C_0^\infty(\Omega)$.  Let $N_1,N_2$ be the smallest nonnegative integers such that
\begin{equation}\label{tail-cutoff-inequalities}
 1-\sigma-\alpha N_1<0,
 \qquad
 1-\sigma-\alpha-\alpha N_2<0.
\end{equation}
Equivalently,
\begin{equation}\label{L-M-choice}
 N_1=\left\lfloor\frac{1-\sigma}{\alpha}\right\rfloor+1,
 \qquad
 N_2=\max\left\{0,
 \left\lfloor\frac{1-\sigma-\alpha}{\alpha}\right\rfloor+1
 \right\}.
\end{equation}
Then the one-sided tails of orders $j\geq N_1$ and $k\geq N_1$, and the mixed tail of total order $j+k\geq N_2$, tend to zero after center averaging.
\end{prop}

\begin{proof}
The floor formulas are valid also when the displayed quotients are integers, because the inequalities in \eqref{tail-cutoff-inequalities} are strict.  The center domination in Lemma~\ref{lem-center-measurability} gives
\begin{equation}\label{tail-center-weight-integrability}
 |\varphi|\omega,
 \quad |\varphi|\widetilde\omega,
 \quad |\varphi|\omega\widetilde\omega
 \in L^1(\Omega).
\end{equation}
We begin with the $L^{p'}$ estimate common to all three tails.  If $r\leq p'$, then $\beta=r/p'$ and interpolation between \eqref{Uj-Lr-general} and \eqref{Uj-Linf-general} gives
\begin{align*}
 \|U_j\|_{L^{p'}(X)}
 &\leq
 \|U_j\|_{L^r(X)}^{\beta}
 \|U_j\|_{L^\infty(X)}^{1-\beta}\\
 &\leq
 C\omega(z_0)\tau^{-\beta/2-(1-\beta)\alpha}\rho_\tau^j
 =C\omega(z_0)\tau^{-\sigma}\rho_\tau^j.
\end{align*}
If $r>p'$, the fixed-support embedding $L^r(X)\hookrightarrow L^{p'}(X)$ gives the same conclusion with $\beta=1$ and $\sigma=1/2$.  Thus
\begin{align}
 \|U_j\|_{L^{p'}(X)}
 &\leq C\omega(z_0)\tau^{-\sigma}\rho_\tau^j,
 \label{Uj-Lpprime-general}\\
 \|\widetilde U_k\|_{L^{p'}(X)}
 &\leq C\widetilde\omega(z_0)\tau^{-\sigma}
          \widetilde\rho_\tau^k.
 \label{tilde-Uj-Lpprime-general}
\end{align}
For all sufficiently large $\tau$, one has $\rho_\tau,\widetilde\rho_\tau\leq1/2$, and hence
\begin{align}
 \sum_{j\geq N_1}\rho_\tau^j
 &\leq 2C_V^{N_1}\tau^{-\alpha N_1},
 &
 \sum_{k\geq N_1}\widetilde\rho_\tau^k
 &\leq 2C_{\widetilde V}^{N_1}\tau^{-\alpha N_1}.
 \label{linear-geometric-tail-bound}
\end{align}
If $\rho_\tau^\ast=\max\{\rho_\tau,\widetilde\rho_\tau\}$, then
\begin{align}
 \sum_{j+k\geq N_2}\rho_\tau^j\widetilde\rho_\tau^k
 &\leq \sum_{n\geq N_2}(n+1)(\rho_\tau^\ast)^n
 \leq C_{N_2}(\rho_\tau^\ast)^{N_2}
 \leq C_{N_2}'\tau^{-\alpha N_2}.
 \label{mixed-geometric-tail-bound}
\end{align}
The estimate remains valid when $N_2=0$.

For the untilded one-sided tail, H\"older's inequality, \eqref{Uj-Lpprime-general}, \eqref{tail-center-weight-integrability}, and \eqref{linear-geometric-tail-bound} give
\begin{align}
 &\tau\int |\varphi(z_0)|
 \sum_{j\geq N_1}
 \left|\int Q(z)U_j(z;z_0,\tau)\dd^2z\right|\dd^2z_0
 \notag\\
 &\qquad\leq C_\varphi\tau^{1-\sigma-\alpha N_1}=o(1).
 \label{left-linear-tail-decay}
\end{align}
The tilded one-sided tail satisfies
\begin{equation}\label{right-linear-tail-decay}
 \tau\int |\varphi(z_0)|
 \sum_{k\geq N_1}
 \left|\int Q(z)\widetilde U_k(z;z_0,\tau)\dd^2z\right|\dd^2z_0
 =o(1).
\end{equation}

For the mixed tail, place the untilded factor in $L^{p'}$ and the tilded factor in $L^\infty$.  Equations \eqref{Uj-Lpprime-general} and \eqref{Uj-Linf-general}, with their tilded counterparts, imply
\begin{equation}\label{termwise-mixed-tail-bound}
 \left|\int Qe^{-i\tau\psi_{z_0}}U_j\widetilde U_k\dd^2z\right|
 \leq
 C\omega(z_0)\widetilde\omega(z_0)
 \tau^{-\sigma-\alpha}\rho_\tau^j\widetilde\rho_\tau^k.
\end{equation}
Therefore
\begin{align}
 &\tau\int |\varphi(z_0)|
 \sum_{j+k\geq N_2}
 \left|\int Qe^{-i\tau\psi_{z_0}}U_j\widetilde U_k\dd^2z\right|
 \dd^2z_0
 \notag\\
 &\qquad\leq C_\varphi\tau^{1-\sigma-\alpha-\alpha N_2}=o(1).
 \label{mixed-tail-decay}
\end{align}
\end{proof}

\subsection{Finite-order reduction}

Let
\begin{equation}\label{remainder-functional-as-born-sums}
 \mathcal R_\tau
 =\sum_{j\geq0}\mathcal B_{j,\tau}
  +\sum_{k\geq0}\widetilde{\mathcal B}_{k,\tau}
  +\sum_{j,k\geq0}\mathcal M_{j,k,\tau}.
\end{equation}
The termwise expansion justified by Lemma~\ref{lem-center-measurability}, together with Proposition~\ref{prop-neumann-tails}, gives the finite-order reduction
\begin{align}
 \mathcal R_\tau
 &={}
 \sum_{0\leq j<N_1}\mathcal B_{j,\tau}
 +\sum_{0\leq k<N_1}\widetilde{\mathcal B}_{k,\tau}
 +\sum_{j+k<N_2}\mathcal M_{j,k,\tau}
 +o(1).
 \label{finite-tail-born-split}
\end{align}
It remains to show that each of these finitely many terms tends to zero.  The one-sided terms will follow from an $L^1$ Riemann--Lebesgue argument and $L^2$ bounds for their output densities.  For mixed terms we also need $L^{r/2}$ duality to treat the product of the two Cauchy transforms.

\section{Averaging the fixed Born terms}\label{sec:fixed-born-orders}\label{subsec-signed-output-estimates}

From this point onward, the Born orders are fixed. Constants may depend on these orders, and it is enough to prove convergence without a rate. For each term in \eqref{finite-tail-born-split}, we retain the finite oscillatory integral rather than estimating each CGO factor separately. This section first identifies the center selected by the phase, then explains the effect of averaging on the terminal differences. The preliminary calculations may be read with smooth coefficients; Proposition~\ref{prop-direct-rough-fixed-order} will justify the same identities for the original $L^p$ coefficients, using the absolute bounds of Section~\ref{sec:positive-densities}.

\subsection{The center selected by the phase}\label{subsec-center-meaning}

A branch is one of the nested Cauchy integrals arising from $U_j$ or $\widetilde U_k$. Its terminal variable is where the input factor is evaluated, and its root is the evaluation variable $x$. In a mixed term the two branches have the same root.

Every one-sided Born phase has total sign $+1$: it starts with $\psi_{z_0}(s)$ and each insertion adds
\[
 \psi_{z_0}(\lambda)-\psi_{z_0}(\eta).
\]
A mixed term has two such branches and the root phase $-\psi_{z_0}(x)$, so its total sign is again $+1$. Thus Lemma~\ref{lem-signed-quadratic-normal-form} applies at every order. In its notation,
\[
 \sum_\nu\varepsilon_\nu\psi_{z_0}(y_\nu)
 =\psi_{z_0}(c)+P,
 \qquad c=\sum_\nu\varepsilon_\nu y_\nu,
 \qquad \sum_\nu\varepsilon_\nu=1.
\]
For fixed spatial variables, $z_0=c$ is the stationary point of the center integral. Accordingly, averaging replaces the center input by $E_\tau$ applied to that input and evaluated at $c$.

We call the corresponding point for a single branch its \emph{signed output}. If the left and right outputs are $u$ and $v$, the mixed signed center is
\begin{equation}\label{model-center-rule}
 c=-x+u+v.
\end{equation}
The terminal points $s,t$ are where $A$ and $\widetilde A$ are evaluated. The root and terminal points need not coincide with $c$, the stationary point selected by center averaging. The residual phase $P$ is independent of $z_0$ and is nondegenerate on each fiber of $c$. On fibers of positive dimension, Lemma~\ref{lem-quadratic-RL} shows that the corresponding oscillatory integral tends to zero for each fixed $L^1$ amplitude. The zero-order one-sided exception is described next.

\subsection{The terminal factors after averaging}\label{subsec-model-born-terms}

Use the denominator convention
\begin{equation}\label{model-denominators}\aliaslabel{cauchy-denominator-conventions}
 D_\partial(x,y)=\bar x-\bar y,
 \qquad D_{\bar\partial}(x,y)=x-y.
\end{equation}
The simplest one-sided term already explains the normalization. At order zero its output is its terminal point $s$, so its averaged terminal factor is
\[
 A(s)E_\tau\varphi(s)-E_\tau(\varphi A)(s).
\]
Replacing $E_\tau$ by the identity would give zero. The actual error is controlled by uniform convergence for $\varphi$ and strong $L^2$ convergence for $\varphi A$. At positive orders the output differs from the terminal, leaving a residual difference; we return to this below.

\paragraph{The lowest mixed term.}
The first mixed term is
\begin{equation}\label{model-M00-def}
 \mathcal M_{0,0,\tau}
 =\frac\tau\pi\int\varphi(z_0)\int
 Q(x)e^{-i\tau\psi_{z_0}(x)}U_0(x)\widetilde U_0(x)
 \,\dd^2x\,\dd^2z_0.
\end{equation}
Here the outputs are $u=s$ and $v=t$, hence
\begin{equation}\label{model-M00-center}
 c=-x+s+t.
\end{equation}
The exact quadratic identity is
\begin{equation}\label{model-M00-normal-form}
 -\psi_{z_0}(x)+\psi_{z_0}(s)+\psi_{z_0}(t)
 =\psi_c(z_0)-2\operatorname{Re}((x-s)(x-t)).
\end{equation}
The kernel without the two terminal differences is
\[
 \frac{Q(x)\chi(s)\chi(t)}
 {\pi^2D_\partial(x,s)D_{\bar\partial}(x,t)}.
\]
Expanding $(A(s)-A(z_0))(\widetilde A(t)-\widetilde A(z_0))$ and averaging the four terms gives
\begin{equation}\label{model-mixed-bracket}\aliaslabel{rough-mixed-born-bracket}
\begin{aligned}
 \mathcal E_\tau(c,s,t)
 &=A(s)\widetilde A(t)E_\tau\varphi(c)
   -A(s)E_\tau(\varphi\widetilde A)(c)\\
 &\quad-\widetilde A(t)E_\tau(\varphi A)(c)
   +E_\tau(\varphi A\widetilde A)(c).
\end{aligned}
\end{equation}
The formal replacement $E_\tau\mapsto I$ gives
\[
 \varphi(c)(A(s)-A(c))(\widetilde A(t)-\widetilde A(c)).
\]
This amplitude is independent of $\tau$ and is not small pointwise. Once it is known to be integrable, the residual quadratic oscillation makes its integral tend to zero.

To separate that amplitude from the averaging errors, set
\begin{equation}\label{model-errors-def}
 \begin{aligned}
 H_1^\tau&=E_\tau\varphi-\varphi,\\
 H_A^\tau&=E_\tau(\varphi A)-\varphi A,\\
 H_{\widetilde A}^\tau&=E_\tau(\varphi\widetilde A)-\varphi\widetilde A,
 \end{aligned}
\end{equation}
\begin{equation}\label{model-product-error-def}\aliaslabel{rough-center-errors-def}
 H_{A\widetilde A}^\tau
 =E_\tau(\varphi A\widetilde A)-\varphi A\widetilde A.
\end{equation}
Then the exact decomposition is
\begin{equation}\label{model-mixed-bracket-decomp}\aliaslabel{rough-mixed-born-bracket-decomp}
\begin{aligned}
 \mathcal E_\tau(c,s,t)
 &=\varphi(c)(A(s)-A(c))(\widetilde A(t)-\widetilde A(c))\\
 &\quad+A(s)\widetilde A(t)H_1^\tau(c)
   -A(s)H_{\widetilde A}^\tau(c)\\
 &\quad-\widetilde A(t)H_A^\tau(c)+H_{A\widetilde A}^\tau(c).
\end{aligned}
\end{equation}
This decomposition is independent of the Born orders. Every additional block changes only the spatial kernel, the signed center, and the residual phase. It introduces no new center input: the four functions in \eqref{model-mixed-bracket} remain the same.

\paragraph{The estimates required by the bracket.}
The first term in \eqref{model-mixed-bracket-decomp} will have an $L^1$ amplitude. The next three terms use
\[
 \|H_1^\tau\|_\infty\longrightarrow0,
 \qquad \|H_A^\tau\|_2+\|H_{\widetilde A}^\tau\|_2\longrightarrow0,
\]
which follow from Lemma~\ref{lem-Etau} and $A,\widetilde A\in L^r$, $r>2$. To pair these errors with the spatial kernel, integrate the absolute kernel over all variables except the signed center. The required spaces for the resulting positive densities are
\begin{center}
\renewcommand{\arraystretch}{1.18}
\begin{tabular}{@{}lll@{}}
\toprule
Terminal weights & Density space at $c$ & Center factor \\
\midrule
$|A(s)|\,|\widetilde A(t)|$ & $L^1$ & $H_1^\tau$ \\
$|A(s)|$ & $L^2$ & $H_{\widetilde A}^\tau$ \\
$|\widetilde A(t)|$ & $L^2$ & $H_A^\tau$ \\
\midrule
$1$ & $L^{(r/2)'}$ & $\varphi A\widetilde A$ \\
\multicolumn{3}{@{}l@{}}{\footnotesize Last row: absolute integrability only.} \\
\bottomrule
\end{tabular}
\end{center}
Lemma~\ref{lem-mixed-center-densities} proves these bounds, and in fact every finite density exponent in the unweighted case. The same bounds prove integrability of the principal amplitude: expand the product of terminal--center differences and pair each placement of $A$ and $\widetilde A$ with the appropriate row.

For the last row only integrability is available, since we know only
\begin{equation}\label{model-product-space}
 \varphi A\widetilde A\in L^{r/2},
 \qquad r=\frac{2p}{2-p},
\end{equation}
and $r/2<2$ when $1<p<4/3$. Section~\ref{subsec-product-error} treats the corresponding error by transposing the average.

\subsection{A positive-order branch: terminal and output}

One inserted block makes the distinction between terminal and output explicit. Suppressing only the universal normalization constant, the left branch is
\begin{multline}\label{model-U1-expanded}
 U_1(x)=C\int e^{i\tau\psi_{z_0}(\lambda)
 -i\tau\psi_{z_0}(\eta)+i\tau\psi_{z_0}(s)}
 \chi(\lambda)q(\eta)\chi(s)(A(s)-A(z_0))\\
 \times\frac{1}
 {D_\partial(x,\lambda)D_{\bar\partial}(\lambda,\eta)
 D_\partial(\eta,s)}\,\dd^2s\,\dd^2\eta\,\dd^2\lambda.
\end{multline}
The terminal is $s$, whereas the output is
\begin{equation}\label{model-left-output}
 u=\lambda-\eta+s.
\end{equation}
Completion of the square gives
\begin{equation}\label{model-left-phase}
 \psi_{z_0}(\lambda)-\psi_{z_0}(\eta)+\psi_{z_0}(s)
 =\psi_u(z_0)+2\operatorname{Re}\bigl((\lambda-\eta)(\eta-s)\bigr).
\end{equation}
Consequently,
\begin{multline}\label{model-U1-output-integral}
 \mathcal B_{1,\tau}
 =C\int e^{2i\tau\operatorname{Re}((\lambda-\eta)(\eta-s))}
 \frac{Q(x)\chi(\lambda)q(\eta)\chi(s)}
 {D_\partial(x,\lambda)D_{\bar\partial}(\lambda,\eta)D_\partial(\eta,s)}\\
 \times\bigl[A(s)E_\tau\varphi(u)-E_\tau(\varphi A)(u)\bigr]
 \,\dd^2x\,\dd^2\lambda\,\dd^2\eta\,\dd^2s.
\end{multline}
The one-sided counterpart of \eqref{model-mixed-bracket-decomp} is
\begin{equation}\label{model-one-sided-bracket}\aliaslabel{rough-left-born-bracket-decomp}\aliaslabel{fixed-left-signed-output-bracket}
\begin{aligned}
 A(s)E_\tau\varphi(u)-E_\tau(\varphi A)(u)
 &=\varphi(u)(A(s)-A(u))\\
 &\quad+A(s)H_1^\tau(u)-H_A^\tau(u).
\end{aligned}
\end{equation}
At positive order the integral of the first term tends to zero by the residual quadratic oscillation. At order zero, $u=s$ and it vanishes identically. The other two terms require only a weighted $L^1$ output density and an unweighted $L^2$ output density. Lemma~\ref{lem-one-sided-output-densities} gives these estimates.

\subsection{Exact branches and phase decomposition}

Insert $j$ blocks in the left branch and $k$ blocks in the right branch. We now give their exact kernels, using \eqref{cauchy-denominator-conventions}. The terminal factor records the normalization; the signed output is determined independently by the phase. Section~\ref{subsubsec-direct-rough-identities} justifies the expansions for $L^p$ coefficients.

\begin{dfn}\label{dfn-exact-born-branches}
\textit{Recursion.}
The left and right Born branches are defined recursively by
\begin{align}
        \mathfrak L_0^{z_0,\tau}[H]
        &=\mathcal C_{\tau,z_0}^{+,\partial}(\chi H),
        \label{exact-left-branch-0}\\
        \mathfrak L_{n+1}^{z_0,\tau}[H]
        &=\mathcal C_{\tau,z_0}^{+,\partial}
          \left(\chi\,
          \mathcal C_{\tau,z_0}^{-,\bar\partial}
          \bigl(q\,\mathfrak L_n^{z_0,\tau}[H]\bigr)\right),
        \label{exact-left-branch-rec}\\
        \mathfrak R_0^{z_0,\tau}[H]
        &=\mathcal C_{\tau,z_0}^{+,\bar\partial}(\chi H),
        \label{exact-right-branch-0}\\
        \mathfrak R_{n+1}^{z_0,\tau}[H]
        &=\mathcal C_{\tau,z_0}^{+,\bar\partial}
          \left(\chi\,
          \mathcal C_{\tau,z_0}^{-,\partial}
          \bigl(\widetilde q\,\mathfrak R_n^{z_0,\tau}[H]\bigr)\right).
        \label{exact-right-branch-rec}
\end{align}
Consequently, for smooth potentials,
\begin{equation}\label{branch-identities-Uj}
        U_j=\mathfrak L_j^{z_0,\tau}[A-A(z_0)],
        \qquad
        \widetilde U_k=
        \mathfrak R_k^{z_0,\tau}[\widetilde A-\widetilde A(z_0)] .
\end{equation}

\par\smallskip\noindent\textit{Left branch.}
Put $\eta_0=x$ and use integration variables
\[
 (\lambda_1,\eta_1),\ldots,(\lambda_j,\eta_j),s .
\]
Its terminal variable is $s$, its signed output is
\begin{equation}\label{exact-left-output}\aliaslabel{model-general-outputs}
        u_j=s+\sum_{\ell=1}^j(\lambda_\ell-\eta_\ell),
\end{equation}
and its phase is
\begin{equation}\label{exact-left-phase}
        \Phi_j^L(z_0)
        =\sum_{\ell=1}^j
          \bigl(\psi_{z_0}(\lambda_\ell)-\psi_{z_0}(\eta_\ell)\bigr)
          +\psi_{z_0}(s).
\end{equation}
After removing the terminal factor $H(s)$ and the oscillation, the exact branch
kernel is
\begin{equation}\label{exact-left-kernel}
\begin{aligned}
\mathcal K_j^L(x;\lambda,\eta,s)
&=\pi^{-(2j+1)}\chi(s)
  \prod_{\ell=1}^j\chi(\lambda_\ell)q(\eta_\ell)\\
&\quad\times
  \frac{1}{D_\partial(\eta_j,s)}
  \prod_{\ell=1}^j
  \frac{1}{D_\partial(\eta_{\ell-1},\lambda_\ell)
          D_{\bar\partial}(\lambda_\ell,\eta_\ell)} .
\end{aligned}
\end{equation}
Here and below the evident interpretation is used when $j=0$: the products are
empty and $\eta_j=\eta_0=x$.  Thus
\begin{equation}\label{exact-left-expanded}
\mathfrak L_j^{z_0,\tau}[H](x)
 =\int e^{i\tau\Phi_j^L(z_0)}
       \mathcal K_j^L(x;\lambda,\eta,s)H(s)\,\dd\Xi_j^L,
\end{equation}
where $\dd\Xi_j^L=\dd^2s\prod_{\ell=1}^j\dd^2\eta_\ell\,\dd^2\lambda_\ell$.

\par\smallskip\noindent\textit{Right branch.}
Put $\nu_0=x$ and use variables
\[
 (\mu_1,\nu_1),\ldots,(\mu_k,\nu_k),t .
\]
Its terminal variable is $t$, its signed output is
\begin{equation}\label{exact-right-output}
        v_k=t+\sum_{m=1}^k(\mu_m-\nu_m),
\end{equation}
and its phase is
\begin{equation}\label{exact-right-phase}
        \Phi_k^R(z_0)
        =\sum_{m=1}^k
          \bigl(\psi_{z_0}(\mu_m)-\psi_{z_0}(\nu_m)\bigr)
          +\psi_{z_0}(t).
\end{equation}
The exact branch kernel is
\begin{equation}\label{exact-right-kernel}
\begin{aligned}
\mathcal K_k^R(x;\mu,\nu,t)
&=\pi^{-(2k+1)}\chi(t)
  \prod_{m=1}^k\chi(\mu_m)\widetilde q(\nu_m)\\
&\quad\times
  \frac{1}{D_{\bar\partial}(\nu_k,t)}
  \prod_{m=1}^k
  \frac{1}{D_{\bar\partial}(\nu_{m-1},\mu_m)
          D_\partial(\mu_m,\nu_m)} ,
\end{aligned}
\end{equation}
and
\begin{equation}\label{exact-right-expanded}
\mathfrak R_k^{z_0,\tau}[H](x)
 =\int e^{i\tau\Phi_k^R(z_0)}
       \mathcal K_k^R(x;\mu,\nu,t)H(t)\,\dd\Xi_k^R,
\end{equation}
where $\dd\Xi_k^R=\dd^2t\prod_{m=1}^k\dd^2\nu_m\,\dd^2\mu_m$.
\end{dfn}

Lemma~\ref{lem-signed-quadratic-normal-form}, applied to
the signs $+,-,\ldots,+,-,+$ gives
\begin{align}
        \Phi_j^L(z_0)
        &=\psi_{z_0}(u_j)+P_j^L,
        \label{exact-left-phase-split}\\
        \Phi_k^R(z_0)
        &=\psi_{z_0}(v_k)+P_k^R,
        \label{exact-right-phase-split}
\end{align}
where
\begin{align}
P_j^L
&=\operatorname{Re}\left(
  \sum_{\ell=1}^j\lambda_\ell^2
 -\sum_{\ell=1}^j\eta_\ell^2+s^2-u_j^2\right),
\label{exact-left-residual-phase}\\
P_k^R
&=\operatorname{Re}\left(
  \sum_{m=1}^k\mu_m^2
 -\sum_{m=1}^k\nu_m^2+t^2-v_k^2\right).
\label{exact-right-residual-phase}
\end{align}
The residual $P_j^L$ is independent of $z_0$ and nondegenerate on every fiber
$u_j=\mathrm{const}$; the same holds for $P_k^R$ on the fibers of $v_k$.

For a mixed term of order $(j,k)$ the two branches meet at the root variable
$x$. Write $\Xi=(x,\lambda,\eta,s,\mu,\nu,t)$ for the complete list of variables. Removing the two terminal differences and the oscillation gives the kernel
\begin{equation}\label{exact-mixed-kernel}
        \mathcal K_{j,k}^{\rm mix}(\Xi)
        =Q(x)\mathcal K_j^L(x;\lambda,\eta,s)
              \mathcal K_k^R(x;\mu,\nu,t),
\end{equation}
and the pre-averaged phase is
\begin{equation}\label{exact-mixed-phase}
        \Theta_{j,k}(z_0;\Xi)
        =-\psi_{z_0}(x)+\Phi_j^L(z_0)+\Phi_k^R(z_0).
\end{equation}
The global signed center is
\begin{equation}\label{exact-mixed-center}
        c_{j,k}=-x+u_j+v_k
        =-x+\sum_{\ell=1}^j(\lambda_\ell-\eta_\ell)+s
           +\sum_{m=1}^k(\mu_m-\nu_m)+t .
\end{equation}
At the core one has
\begin{equation}\label{exact-core-phase-split}
-\psi_{z_0}(x)+\psi_{z_0}(u_j)+\psi_{z_0}(v_k)
 =\psi_{z_0}(c_{j,k})+P^{\rm core}(x,u_j,v_k),
\end{equation}
where
\begin{equation}\label{exact-core-residual-phase}
P^{\rm core}(x,u,v)
 =\operatorname{Re}\bigl(-x^2+u^2+v^2-(-x+u+v)^2\bigr).
\end{equation}
Combining the two branch identities with the core identity yields the
hierarchical phase decomposition
\begin{equation}\label{exact-mixed-phase-split}
        \Theta_{j,k}(z_0;\Xi)
        =\psi_{z_0}(c_{j,k})+P_{j,k}(\Xi),
        \qquad
        P_{j,k}=P_j^L+P_k^R+P^{\rm core}(x,u_j,v_k).
\end{equation}
The residual is independent of $z_0$.  Its nondegeneracy on each fiber
$c_{j,k}=\mathrm{const}$ follows by applying
Lemma~\ref{lem-signed-quadratic-normal-form} once to the full signed list of
variables in \eqref{exact-mixed-phase}.

For later reference, when $j=k=1$ and
$(\lambda,\eta)=(\lambda_1,\eta_1)$,
$(\mu,\nu)=(\mu_1,\nu_1)$, formula \eqref{exact-mixed-kernel} reads
\begin{align}
\mathcal K_{1,1}^{\rm mix}
&=\pi^{-6}Q(x)\chi(s)\chi(t)\chi(\lambda)q(\eta)
  \chi(\mu)\widetilde q(\nu)                                  \notag\\
&\quad\times
\frac{1}{D_\partial(\eta,s)D_\partial(x,\lambda)
        D_{\bar\partial}(\lambda,\eta)}
\frac{1}{D_{\bar\partial}(\nu,t)D_{\bar\partial}(x,\mu)
        D_\partial(\mu,\nu)} .                       \label{exact-mixed-kernel-11}
\end{align}
The corresponding phase and center are
\begin{align}
\Theta_{1,1}
&=-\psi_{z_0}(x)+\psi_{z_0}(\lambda)-\psi_{z_0}(\eta)
  +\psi_{z_0}(s)+\psi_{z_0}(\mu)-\psi_{z_0}(\nu)
  +\psi_{z_0}(t),                                      \label{exact-mixed-phase-11}\\
 c_{1,1}
&=-x+\lambda-\eta+s+\mu-\nu+t .                         \label{exact-mixed-center-11}
\end{align}

\section{Positive densities and fixed-order identities}\label{sec:positive-densities}

We prove the estimates announced by the averaged brackets. The densities below integrate out every variable except a branch output or the mixed signed center. They retain the absolute kernel and a specified terminal weight, but no oscillation. Their purpose is to bound a function evaluated at the output in its own Lebesgue norm. All constants may depend on the fixed orders; no estimate uniform in the Born order is needed here.

\subsection{Push-forward to the branch output}
\label{subsubsec-positive-output-densities}

For a nonnegative terminal weight $F$, the density $\Gamma_j^F(x,u)$ is characterized by the push-forward identity \eqref{exact-left-output-pushforward}: integrating $h(u)$ against it gives the absolute branch integral with factors $F(s)h(u_j)$. The following recursion constructs that density directly. The shift in its output variable is the same shift contributed by an inserted block to the phase.

Let $F:\C\to[0,\infty]$ be measurable.  Define the positive left-output
density recursively by
\begin{align}
\Gamma_0^F(x,u)
&=\pi^{-1}k_R(x-u)|\chi(u)|F(u),
\label{positive-left-output-0}\\
\Gamma_{j+1}^F(x,u)
&=\pi^{-2}\iint
 k_R(x-\lambda)|\chi(\lambda)|
 k_R(\lambda-\eta)|q(\eta)| \notag\\
&\qquad\qquad\times
 \Gamma_j^F(\eta,u-\lambda+\eta)
 \dd^2\lambda\dd^2\eta .
\label{positive-left-output-rec}
\end{align}
The translated output $u-\lambda+\eta$ is forced by the identity
\[
 u=(\lambda-\eta)+u_{\mathrm{inner}}.
\]
The positive right-output density is defined analogously by
\begin{align}
\widetilde\Gamma_0^G(x,v)
&=\pi^{-1}k_R(x-v)|\chi(v)|G(v),
\label{positive-right-output-0}\\
\widetilde\Gamma_{k+1}^G(x,v)
&=\pi^{-2}\iint
 k_R(x-\mu)|\chi(\mu)|
 k_R(\mu-\nu)|\widetilde q(\nu)| \notag\\
&\qquad\qquad\times
 \widetilde\Gamma_k^G(\nu,v-\mu+\nu)
 \dd^2\mu\dd^2\nu .
\label{positive-right-output-rec}
\end{align}
After absolute values the two Cauchy orientations give the same localized
kernel, which is why the recursions have the same form.

\begin{lem}
\label{lem-exact-output-pushforward}
For every $j,k\ge0$, all nonnegative measurable terminal weights $F,G$, every
nonnegative Borel function $h$, and every root $x$ in the fixed compact set $X_{j,k}$ chosen above, which
contains all possible branch roots and $\operatorname{supp}Q$, one has
\begin{align}
\int h(u)\Gamma_j^F(x,u)\dd^2u
&=\int
 |\mathcal K_j^L(x;\lambda,\eta,s)|F(s)h(u_j)
 \dd\Xi_j^L,
\label{exact-left-output-pushforward}\\
\int h(v)\widetilde\Gamma_k^G(x,v)\dd^2v
&=\int
 |\mathcal K_k^R(x;\mu,\nu,t)|G(t)h(v_k)
 \dd\Xi_k^R.
\label{exact-right-output-pushforward}
\end{align}
The identities are understood in $[0,\infty]$.  In particular,
$\Gamma_j^F(x,\cdot)\dd^2u$ and
$\widetilde\Gamma_k^G(x,\cdot)\dd^2v$ are exactly the push-forwards of
the corresponding absolute branch measures under the signed-output maps
\eqref{exact-left-output} and \eqref{exact-right-output}.
\end{lem}

\begin{proof}
We prove \eqref{exact-left-output-pushforward}; the right-hand identity is the
same after replacing $(\lambda,\eta,q)$ by
$(\mu,\nu,\widetilde q)$.  For $j=0$, the assertion follows immediately
from \eqref{positive-left-output-0}, since $u_0=s$ and
\[
 |\mathcal K_0^L(x;s)|=\pi^{-1}k_R(x-s)|\chi(s)|
\]
on the fixed compact set under consideration.

Assume the identity at order $j$.  Tonelli's theorem and
\eqref{positive-left-output-rec} give
\begin{align*}
\int h(u)\Gamma_{j+1}^F(x,u)\dd^2u
&=\pi^{-2}\iiint
 k_R(x-\lambda)|\chi(\lambda)|
 k_R(\lambda-\eta)|q(\eta)|\,\notag\\
&\qquad\qquad\times
 h(u)\Gamma_j^F(\eta,u-\lambda+\eta)
 \dd^2u\dd^2\lambda\dd^2\eta.
\end{align*}
Set $w=u-\lambda+\eta$.  Translation invariance of Lebesgue measure changes
the inner integral to
\[
 \int h(w+\lambda-\eta)\Gamma_j^F(\eta,w)\dd^2w.
\]
Apply the induction hypothesis at the root $\eta$ to $w\mapsto h(w+\lambda-\eta)$.  The inner branch has output $w$, and the new outer block contributes $\lambda-\eta$.  Their sum is
\[
 w+\lambda-\eta
 =s+\sum_{\ell=1}^{j+1}(\lambda_\ell-\eta_\ell)
 =u_{j+1},
\]
after reindexing the inner variables.  The two new localized Cauchy kernels, together with $|\chi(\lambda)q(\eta)|$ and $\pi^{-2}$, form exactly the absolute value of the new outer block in \eqref{exact-left-kernel}.  This proves \eqref{exact-left-output-pushforward} at order $j+1$.
\end{proof}

These identities use no oscillatory cancellation. All roots arising in the induction, and in particular the roots used in the
one-sided and mixed Born integrals, lie in the indicated compact set; the densities themselves are defined for
all $x\in\C$, which is convenient for the uniform norm estimates below.

\subsection{Born insertion preserves the output exponents}

For $1\le a<\infty$ set
\begin{equation}\label{branch-output-Ma}
 \mathcal N_a(j,F)=\sup_{x\in\C}
        \|\Gamma_j^F(x,\cdot)\|_{L^a(\C)},
\end{equation}
with the analogous notation $\widetilde{\mathcal N}_a(k,G)$ for the right
branch.  For example, the first nontrivial left density is
\[
 \Gamma_1^F(x,u)=\pi^{-2}\iint
 k_R(x-\lambda)|\chi(\lambda)|k_R(\lambda-\eta)|q(\eta)|
 \Gamma_0^F(\eta,u-\lambda+\eta)
 \dd^2\lambda\dd^2\eta.
\]

\begin{lem}
\label{lem-branch-output-bounds}
Let $1<p<2$ and
\[
        r=\frac{2p}{2-p}.
\]
For every fixed $j,k\ge0$ the following estimates hold.
\begin{align}
 \mathcal N_a(j,1)
 &\le C_{j,a}\|q\|_{L^p}^j,
 &&1\le a<2,                                      \label{branch-output-unweighted-left}\\
 \widetilde{\mathcal N}_a(k,1)
 &\le C_{k,a}\|\widetilde q\|_{L^p}^k,
 &&1\le a<2,                                      \label{branch-output-unweighted-right}\\
 \mathcal N_a(j,F)
 &\le C_{j,a}\|q\|_{L^p}^j\|F\|_{L^r},
 &&F\in L^r(\C),\quad 1\le a<p,                 \label{branch-output-weighted-left}\\
 \widetilde{\mathcal N}_a(k,G)
 &\le C_{k,a}\|\widetilde q\|_{L^p}^k\|G\|_{L^r},
 &&G\in L^r(\C),\quad 1\le a<p.                \label{branch-output-weighted-right}
\end{align}
For each fixed order, the supports of
$\Gamma_j^F(x,\cdot)$ and $\widetilde\Gamma_k^G(x,\cdot)$ are contained in
fixed compact sets, uniformly in $x$ and in the terminal weights.  The constants
depend on the displayed exponents, the fixed order, $R$, and the fixed cutoffs,
but not on $x,F,G$.
\end{lem}

\begin{proof}
We prove the two left-branch estimates.  The right branch is identical after
replacing $q$ by $\widetilde q$.

First note the support assertion.  At order zero the factor $\chi(u)$ confines
the output to $\operatorname{supp}\chi$.  If the order-$j$ output is supported
in a compact set $Y_j$, then a nonzero term in
\eqref{positive-left-output-rec} has
\[
        u-\lambda+\eta\in Y_j,
        \qquad \lambda\in\operatorname{supp}\chi,
        \qquad \eta\in\operatorname{supp}q.
\]
Hence
\[
        u\in Y_j+\operatorname{supp}\chi-\operatorname{supp}q,
\]
which is again compact and independent of $x$ and $F$.

Let $1\le a<\infty$.  Minkowski's integral inequality, translation invariance
of the $L^a$ norm, and \eqref{positive-left-output-rec} give
\begin{align*}
 \|\Gamma_{j+1}^F(x,\cdot)\|_{L^a}
 &\le \pi^{-2}\mathcal N_a(j,F)
 \iint k_R(x-\lambda)|\chi(\lambda)|
        k_R(\lambda-\eta)|q(\eta)|
        \dd^2\lambda\dd^2\eta \\
 &\le C\mathcal N_a(j,F)(K_R*|q|)(x).
\end{align*}
Since $p>1$, one has $p'<\infty$, and
$K_R\in L^{p'}(\C)$ by \eqref{kernel-basic-facts-general}.  Therefore
\begin{equation}\label{branch-output-recursive-bound}
 \mathcal N_a(j+1,F)
 \le C\|K_R\|_{L^{p'}}\|q\|_{L^p}\mathcal N_a(j,F).
\end{equation}

For the unweighted base case, if $1\le a<2$, then
\begin{equation}\label{branch-output-base-unweighted}
 \mathcal N_a(0,1)
 \le \pi^{-1}\|\chi\|_{L^\infty}\|k_R\|_{L^a}<\infty.
\end{equation}
Iteration of \eqref{branch-output-recursive-bound} proves
\eqref{branch-output-unweighted-left}.

Now let $F\in L^r$ and $1\le a<p$.  Choose $\gamma$ by
\begin{equation}\label{branch-output-gamma-choice}
        \frac1a=\frac1\gamma+\frac1r.
\end{equation}
Since
\[
        \frac1p=\frac12+\frac1r,
\]
the condition $a<p$ is equivalent to $1/\gamma>1/2$, while $a\ge1$ implies
$1/\gamma<1$.  Thus $1<\gamma<2$, so $k_R\in L^\gamma$.  H\"older's inequality
in the output variable yields
\begin{equation}\label{branch-output-base-weighted}
 \mathcal N_a(0,F)
 \le \pi^{-1}\|\chi\|_{L^\infty}
       \|k_R\|_{L^\gamma}\|F\|_{L^r}.
\end{equation}
Combining this with \eqref{branch-output-recursive-bound} gives
\eqref{branch-output-weighted-left}.
\end{proof}

In particular, taking $a=1$ gives the absolute branch-mass bounds
\begin{align}
 \sup_x\int\Gamma_j^1(x,u)\dd^2u
 &\le C_j\|q\|_{L^p}^j,                              \label{branch-output-mass-unweighted}\\
 \sup_x\int\Gamma_j^F(x,u)\dd^2u
 &\le C_j\|q\|_{L^p}^j\|F\|_{L^r},                 \label{branch-output-mass-weighted}
\end{align}
and the analogous right-branch estimates.  The interval $1\le a<p$ is
nonempty precisely because $p>1$; at the endpoint $p=1$ the argument would also
require the false bound $K_R\in L^\infty$.

\subsection{Densities for the one-sided and mixed terms}

\paragraph{One-sided densities.}
\label{subsubsec-one-sided-output-densities}

For a one-sided term, integrate the root against $|Q|$. The weighted $L^1$ bound follows directly from the branch mass. The unweighted $L^2$ bound needs a further step: there is no second branch to convolve with, so the proof exposes the terminal Cauchy edge and uses its smoothing. For nonnegative terminal weights $F,G$, define
\begin{align}
 J_j^F(u)&=\int_{\C}|Q(x)|\Gamma_j^F(x,u)\dd^2x,
 \label{one-sided-left-density-def}\\
 \widetilde J_k^G(v)&=\int_{\C}|Q(x)|
             \widetilde\Gamma_k^G(x,v)\dd^2x.
 \label{one-sided-right-density-def}
\end{align}
By Lemma~\ref{lem-exact-output-pushforward}, these are exactly the positive
push-forwards of the absolute one-sided Born kernels to the signed output after
the root has been paired with $|Q|$.

\begin{lem}
\label{lem-one-sided-output-densities}
Let $1<p<2$ and $r=2p/(2-p)$.  For every fixed $j,k\ge0$ and every
nonnegative $F,G\in L^r(\C)$,
\begin{align}
 \|J_j^F\|_{L^1}
 &\le C_j\|Q\|_{L^p}\|q\|_{L^p}^j\|F\|_{L^r},
 \label{one-sided-left-weighted-L1}\\
 \|\widetilde J_k^G\|_{L^1}
 &\le C_k\|Q\|_{L^p}\|\widetilde q\|_{L^p}^k
       \|G\|_{L^r},
 \label{one-sided-right-weighted-L1}\\
 \|J_j^1\|_{L^2}
 &\le C_j\|Q\|_{L^p}\|q\|_{L^p}^j,
 \label{one-sided-left-unweighted-L2}\\
 \|\widetilde J_k^1\|_{L^2}
 &\le C_k\|Q\|_{L^p}\|\widetilde q\|_{L^p}^k.
 \label{one-sided-right-unweighted-L2}
\end{align}
Each density is supported in a fixed compact set depending only on its order
and on the fixed supports and cutoffs.
\end{lem}

\begin{proof}
We prove the left-branch statements.  The right-branch proof is identical.
The support assertion follows at once from the support statement in
Lemma~\ref{lem-branch-output-bounds}.

Since $Q$ is supported in a fixed bounded set, H\"older's inequality and
\eqref{branch-output-mass-weighted} give
\begin{align*}
 \|J_j^F\|_{L^1}
 &=\int |Q(x)|\left(\int\Gamma_j^F(x,u)\dd^2u\right)\dd^2x\\
 &\le C_j\|Q\|_{L^1}\|q\|_{L^p}^j\|F\|_{L^r}\\
 &\le C_j\|Q\|_{L^p}\|q\|_{L^p}^j\|F\|_{L^r}.
\end{align*}
This proves \eqref{one-sided-left-weighted-L1}.

We turn to \eqref{one-sided-left-unweighted-L2}.  At order zero,
\begin{equation}\label{one-sided-order-zero-pointwise}
 J_0^1(u)
 =\pi^{-1}|\chi(u)|(k_R*|Q|)(u).
\end{equation}
The Hardy--Littlewood--Sobolev inequality gives
$k_R*|Q|\in L^r$, where $r>2$.  Since $\chi$ has compact support,
\begin{equation}\label{one-sided-order-zero-L2}
 \|J_0^1\|_{L^2}\le C\|k_R*|Q|\|_{L^r}
 \le C\|Q\|_{L^p}.
\end{equation}

Now suppose $j\ge1$ and set
\begin{equation}\label{one-sided-terminal-B-def}
 \beta_R=k_R*|q|.
\end{equation}
Then $\beta_R$ is supported in a fixed compact set and, by the
Hardy--Littlewood--Sobolev inequality,
\begin{equation}\label{one-sided-terminal-B-bound}
 \|\beta_R\|_{L^r}\le C\|q\|_{L^p}.
\end{equation}
We claim that
\begin{equation}\label{one-sided-terminal-collapse}
 \Gamma_j^1(x,u)
 \le C\int k_R(y-u)\Gamma_{j-1}^{\beta_R}(x,y)\dd^2y
 \qquad\text{for a.e. }(x,u).
\end{equation}
Indeed, in the exact order-$j$ branch let
\begin{equation}\label{one-sided-penultimate-output}
 y=\lambda_j+\sum_{\ell=1}^{j-1}
             (\lambda_\ell-\eta_\ell).
\end{equation}
Then
\begin{equation}\label{one-sided-terminal-output-relation}
 u_j=y-\eta_j+s,
 \qquad
 k_R(\eta_j-s)=k_R(y-u_j).
\end{equation}
Use Lemma~\ref{lem-exact-output-pushforward} and test against an arbitrary
nonnegative Borel function of $u_j$.  After the determinant-one substitution
$s=u-y+\eta_j$, the last terminal kernel becomes $k_R(y-u)$, while
$|\chi(s)|$ is bounded.  Integrating the remaining last coefficient edge gives
\[
 \int k_R(\lambda_j-\eta_j)|q(\eta_j)|\dd^2\eta_j
 =\beta_R(\lambda_j).
\]
The preceding variables now form an order-$(j-1)$ branch with terminal variable $\lambda_j$, terminal weight $\beta_R(\lambda_j)$, and output $y$.  This proves \eqref{one-sided-terminal-collapse} as an inequality of positive measures, and hence almost everywhere.

Define
\begin{equation}\label{one-sided-Hj-def}
 \mathcal H_j(y)=\int |Q(x)|\Gamma_{j-1}^{\beta_R}(x,y)\dd^2x.
\end{equation}
Choose once and for all an exponent $a$ with
\begin{equation}\label{one-sided-a-choice}
 1<a<p
\end{equation}
and define $a^*$ by
\begin{equation}\label{one-sided-astar-def}
 \frac1{a^*}=\frac1a-\frac12.
\end{equation}
Minkowski's inequality, the bounded support of $Q$,
\eqref{branch-output-weighted-left}, and
\eqref{one-sided-terminal-B-bound} give
\begin{align}
 \|\mathcal H_j\|_{L^a}
 &\le \|Q\|_{L^1}\mathcal N_a(j-1,\beta_R) \notag\\
 &\le C_j\|Q\|_{L^p}\|q\|_{L^p}^{j-1}\|\beta_R\|_{L^r}
 \le C_j\|Q\|_{L^p}\|q\|_{L^p}^{j}.
 \label{one-sided-Hj-La}
\end{align}
Integrating \eqref{one-sided-terminal-collapse} against $|Q(x)|$ yields
\begin{equation}\label{one-sided-Jj-Hj}
 J_j^1\le C k_R*\mathcal H_j.
\end{equation}
Because $1<a<2$, the Hardy--Littlewood--Sobolev inequality and
\eqref{one-sided-Hj-La} imply
\begin{equation}\label{one-sided-Jj-Lastar}
 \|J_j^1\|_{L^{a^*}}
 \le C_j\|Q\|_{L^p}\|q\|_{L^p}^j.
\end{equation}
Since $a>1$, we have $a^*>2$.  The fixed compact support of $J_j^1$ then gives the embedding $L^{a^*}\hookrightarrow L^2$, which proves \eqref{one-sided-left-unweighted-L2}.
\end{proof}

\paragraph{Mixed densities at the global center.}
\label{subsubsec-mixed-center-densities}

For a mixed term, the root and the two branch outputs occur together only through
\[
        c=-x+u+v.
\]
Thus the push-forward of the absolute kernel to the global center is the convolution of the two branch-output densities.  For nonnegative terminal weights $F,G$, define
\begin{equation}\label{mixed-center-density-def}
 I_{j,k}^{F,G}(c)
 =\int_{\C}|Q(x)|
   \bigl(\Gamma_j^F(x,\cdot)
         *\widetilde\Gamma_k^G(x,\cdot)\bigr)(c+x)
   \dd^2x .
\end{equation}
Equivalently,
\begin{equation}\label{mixed-center-density-expanded}
 I_{j,k}^{F,G}(c)
 =\iint_{\C^2}|Q(x)|\Gamma_j^F(x,u)
       \widetilde\Gamma_k^G(x,c+x-u)
       \dd^2u\dd^2x .
\end{equation}
The exact push-forward identity and the required norm estimates are as follows.

\begin{lem}
\label{lem-mixed-center-densities}
Let $1<p<2$ and $r=2p/(2-p)$.  For every fixed $j,k\ge0$, all
nonnegative measurable terminal weights $F,G$, and every nonnegative Borel
function $h$, one has
\begin{align}
 \int h(c)I_{j,k}^{F,G}(c)\dd^2c
 &=\int |Q(x)|\,
       |\mathcal K_j^L(x;\lambda,\eta,s)|
       |\mathcal K_k^R(x;\mu,\nu,t)| \notag\\
 &\qquad\times F(s)G(t)h(c_{j,k})
       \dd^2x\dd\Xi_j^L\dd\Xi_k^R .
 \label{exact-mixed-center-pushforward}
\end{align}
Thus $I_{j,k}^{F,G}(c)\dd^2c$ is exactly the push-forward of the
absolute weighted mixed Born measure under the global center map
$c_{j,k}=-x+u_j+v_k$.

More generally, if $1\le a,b<\infty$ and $1\le \kappa\le\infty$ satisfy
\begin{equation}\label{mixed-young-exponents}
        1+\frac1\kappa=\frac1a+\frac1b,
\end{equation}
and the two branch norms on the right are finite, then
\begin{equation}\label{mixed-center-general-young}
 \|I_{j,k}^{F,G}\|_{L^\kappa}
 \le \|Q\|_{L^1}
      \mathcal N_a(j,F)\,
      \widetilde{\mathcal N}_b(k,G).
\end{equation}
In particular, with constants depending on the fixed orders and the displayed
exponents,
\begin{align}
 \|I_{j,k}^{1,1}\|_{L^\kappa}
 &\le C_{j,k,\kappa}\|Q\|_{L^p}
       \|q\|_{L^p}^j\|\widetilde q\|_{L^p}^k,
 &&1\le \kappa<\infty,                                      
 \label{mixed-center-unweighted-Lt}\\
 \|I_{j,k}^{F,1}\|_{L^2}
 &\le C_{j,k}\|Q\|_{L^p}
       \|q\|_{L^p}^j\|\widetilde q\|_{L^p}^k
       \|F\|_{L^r},
 &&F\in L^r(\C),                                      
 \label{mixed-center-left-weighted-L2}\\
 \|I_{j,k}^{1,G}\|_{L^2}
 &\le C_{j,k}\|Q\|_{L^p}
       \|q\|_{L^p}^j\|\widetilde q\|_{L^p}^k
       \|G\|_{L^r},
 &&G\in L^r(\C),                                      
 \label{mixed-center-right-weighted-L2}\\
 \|I_{j,k}^{F,G}\|_{L^1}
 &\le C_{j,k}\|Q\|_{L^p}
       \|q\|_{L^p}^j\|\widetilde q\|_{L^p}^k
       \|F\|_{L^r}\|G\|_{L^r},
 &&F,G\in L^r(\C).                                   
 \label{mixed-center-two-weighted-L1}
\end{align}
Each of these center densities has fixed compact support, depending only on the
orders and on the fixed supports and cutoffs.
\end{lem}

\begin{proof}
By \eqref{mixed-center-density-expanded}, Tonelli's theorem, and the
translation $v=c+x-u$,
\begin{align*}
 \int h(c)I_{j,k}^{F,G}(c)\dd^2c
 &=\iiint |Q(x)|\Gamma_j^F(x,u)
       \widetilde\Gamma_k^G(x,v) \\
 &\qquad\qquad\times h(-x+u+v)
       \dd^2v\dd^2u\dd^2x.
\end{align*}
Apply Lemma~\ref{lem-exact-output-pushforward} first to the right branch with
$v\mapsto h(-x+u+v)$ and then to the left branch.  Since
$-x+u_j+v_k=c_{j,k}$, this gives
\eqref{exact-mixed-center-pushforward}.

For the norm estimate, Minkowski's integral inequality and translation
invariance in the center variable give
\begin{align*}
 \|I_{j,k}^{F,G}\|_{L^\kappa_c}
 &\le \int |Q(x)|
       \|\Gamma_j^F(x,\cdot)
          *\widetilde\Gamma_k^G(x,\cdot)\|_{L^\kappa}
       \dd^2x.
\end{align*}
Young's convolution inequality with
\eqref{mixed-young-exponents} therefore yields
\eqref{mixed-center-general-young}.

For \eqref{mixed-center-unweighted-Lt}, choose
\begin{equation}\label{mixed-unweighted-exponent-choice}
        a=b=\frac{2\kappa}{\kappa+1}.
\end{equation}
For every finite $\kappa\ge1$, this exponent belongs to $[1,2)$, so the unweighted
branch estimates \eqref{branch-output-unweighted-left} and
\eqref{branch-output-unweighted-right} apply.  Since $Q$ has fixed compact
support, $\|Q\|_{L^1}\le C\|Q\|_{L^p}$.

For \eqref{mixed-center-left-weighted-L2}, take
\begin{equation}\label{mixed-one-weight-exponent-choice}
        a=\frac{2p}{p+1},
        \qquad
        b=\frac{2p}{2p-1}.
\end{equation}
Then
\[
        1<a<p,
        \qquad
        1<b<2,
        \qquad
        \frac1a+\frac1b=\frac32.
\]
Hence \eqref{mixed-center-general-young},
\eqref{branch-output-weighted-left}, and
\eqref{branch-output-unweighted-right} give the assertion.  Interchanging the
two branches proves \eqref{mixed-center-right-weighted-L2}.  Finally,
\eqref{mixed-center-two-weighted-L1} follows from
\eqref{mixed-center-general-young} with $a=b=\kappa=1$ and the two weighted branch
estimates.  The support statement follows from the fixed-order compact support
of both branch-output densities and the fixed support of $Q$.
\end{proof}

These are the density bounds listed in Section~\ref{subsec-model-born-terms}; they will be used both to prove integrability of the principal amplitudes and to estimate the averaging errors.

\subsection{Absolute convergence and the fixed-order identities}
\label{subsubsec-direct-rough-identities}

We now justify the calculations of Section~\ref{sec:fixed-born-orders} without smoothing the coefficients. The branch masses give absolute convergence before averaging; the output-density bounds justify each term after averaging. Kernel truncation and dominated convergence then identify the multiple integrals with the original recursive Cauchy operators.

\begin{lem}
\label{lem-fixed-order-absolute-convergence}
Let $1<p<2$, let $r=2p/(2-p)$, and fix $j,k\ge0$.  For almost every center
$z_0$ and every root $x$ in the fixed compact set containing the supports and
cutoffs, one has
\begin{align}
&\int |\mathcal K_j^L(x;\lambda,\eta,s)|
       |A(s)-A(z_0)|\dd\Xi_j^L
 \le C_j\|q\|_{L^p}^j
       \bigl(\|A\|_{L^r}+|A(z_0)|\bigr),
\label{rough-left-branch-mass}\\
&\int |\mathcal K_k^R(x;\mu,\nu,t)|
       |\widetilde A(t)-\widetilde A(z_0)|\dd\Xi_k^R
 \le C_k\|\widetilde q\|_{L^p}^k
       \bigl(\|\widetilde A\|_{L^r}
             +|\widetilde A(z_0)|\bigr).
\label{rough-right-branch-mass}
\end{align}
Consequently,
\begin{align}
&\int |Q(x)|\,|\mathcal K_j^L(x;\lambda,\eta,s)|
       |A(s)-A(z_0)|\dd^2x\dd\Xi_j^L
\notag\\
&\qquad\le C_j\|Q\|_{L^1}\|q\|_{L^p}^j
       \bigl(\|A\|_{L^r}+|A(z_0)|\bigr),
\label{rough-left-root-mass}\\
&\int |Q(x)|\,|\mathcal K_j^L|\,|\mathcal K_k^R|
       |A(s)-A(z_0)|
       |\widetilde A(t)-\widetilde A(z_0)|
       \dd^2x\dd\Xi_j^L\dd\Xi_k^R
\notag\\
&\qquad\le C_{j,k}\|Q\|_{L^1}
       \|q\|_{L^p}^j\|\widetilde q\|_{L^p}^k
       \bigl(\|A\|_{L^r}+|A(z_0)|\bigr)
       \bigl(\|\widetilde A\|_{L^r}
             +|\widetilde A(z_0)|\bigr).
\label{rough-mixed-root-mass}
\end{align}
The right-hand sides of \eqref{rough-left-root-mass} and
\eqref{rough-mixed-root-mass} are integrable in $z_0$ after multiplication
by $|\varphi(z_0)|$.  The analogous assertion holds for the right one-sided
branch.
\end{lem}

\begin{proof}
By Lemma~\ref{lem-exact-output-pushforward} with $h=1$,
\begin{align*}
\int |\mathcal K_j^L|\,|A(s)-A(z_0)|\dd\Xi_j^L
&\le \int \Gamma_j^{|A|}(x,u)\dd^2u
   +|A(z_0)|\int\Gamma_j^1(x,u)\dd^2u.
\end{align*}
The estimates \eqref{branch-output-weighted-left} and
\eqref{branch-output-unweighted-left}, both with exponent $a=1$, give
\eqref{rough-left-branch-mass}.  The right-hand estimate is identical.
Integrating \eqref{rough-left-branch-mass} against $|Q(x)|$ gives
\eqref{rough-left-root-mass}.  For the mixed estimate, apply
\eqref{rough-left-branch-mass} and
\eqref{rough-right-branch-mass} separately at the common root $x$, multiply
the two bounds, and integrate against $|Q(x)|$.

For the center integration, observe that on the compact support of $\varphi$, the functions $A$ and $\widetilde A$ belong to $L^r$, hence to
$L^1$, while
$A\widetilde A\in L^{r/2}\subset L^1$ because $r>2$.  Expanding the product
on the right of \eqref{rough-mixed-root-mass} therefore gives an integrable
majorant.  The one-sided case follows in the same way without the product term.
\end{proof}

We can therefore state the fixed-order formulas directly for the original $L^p$ coefficients.
The four errors $H_1^\tau$, $H_A^\tau$, $H_{\widetilde A}^\tau$, and
$H_{A\widetilde A}^\tau$ were defined with the common averaged bracket,
equations~\eqref{model-errors-def}--\eqref{model-product-error-def}.
The last is understood through the integral formula~\eqref{Etau-def}; its
input is compactly supported and belongs to $L^1$.

\begin{prop}
\label{prop-direct-rough-fixed-order}
Let $1<p<2$, fix $j,k\ge0$, and let $\tau\ge2$.  The branch expansions
\eqref{exact-left-expanded} and \eqref{exact-right-expanded}, with
$H=A-A(z_0)$ and
$H=\widetilde A-\widetilde A(z_0)$ respectively, hold for almost every
$(x,z_0)$ for the original $L^p$ coefficients.  Moreover, the one-sided Born
functionals satisfy
\begin{equation}\label{rough-left-born-output-identity}\aliaslabel{fixed-left-signed-output-integral}
\mathcal B_{j,\tau}
=\int e^{i\tau P_j^L}
 Q(x)\mathcal K_j^L(x;\lambda,\eta,s)
 \mathcal E_\tau^L(u_j,s)
 \dd^2x\dd\Xi_j^L,
\end{equation}
\begin{equation}\label{rough-right-born-output-identity}
\widetilde{\mathcal B}_{k,\tau}
=\int e^{i\tau P_k^R}
 Q(x)\mathcal K_k^R(x;\mu,\nu,t)
 \mathcal E_\tau^R(v_k,t)
 \dd^2x\dd\Xi_k^R,
\end{equation}
where
\begin{align}
\mathcal E_\tau^L(u,s)
&=A(s)E_\tau\varphi(u)-E_\tau(\varphi A)(u),
\label{rough-left-born-bracket}\\
\mathcal E_\tau^R(v,t)
&=\widetilde A(t)E_\tau\varphi(v)
  -E_\tau(\varphi\widetilde A)(v).
\label{rough-right-born-bracket}
\end{align}
The left bracket has the decomposition \eqref{rough-left-born-bracket-decomp}. The right counterpart is
\begin{align}
\mathcal E_\tau^R(v,t)
&=\varphi(v)(\widetilde A(t)-\widetilde A(v))
  +\widetilde A(t)H_1^\tau(v)-H_{\widetilde A}^\tau(v).
\label{rough-right-born-bracket-decomp}
\end{align}

The mixed functional satisfies
\begin{equation}
\mathcal M_{j,k,\tau}
=\int e^{i\tau P_{j,k}(\Xi)}
 \mathcal K_{j,k}^{\rm mix}(\Xi)
 \mathcal E_\tau(c_{j,k},s,t)
 \dd^2x\dd\Xi_j^L\dd\Xi_k^R,
\label{rough-mixed-born-output-identity}\aliaslabel{fixed-mixed-signed-output-integral}
\end{equation}
where $\mathcal E_\tau(c,s,t)$ is the order-independent bracket in \eqref{rough-mixed-born-bracket}, with the exact decomposition \eqref{rough-mixed-born-bracket-decomp}.
Every integral in
\eqref{rough-left-born-output-identity}--
\eqref{rough-mixed-born-output-identity} is absolutely convergent for fixed
$\tau$; the oscillatory factors have modulus one.
\end{prop}

\begin{proof}
For $\delta>0$, replace every Cauchy factor by its truncated version
\begin{equation}
 \frac{\mathbf 1_{\{|z-w|>\delta\}}}{D_\partial(z,w)},
 \qquad
 \frac{\mathbf 1_{\{|z-w|>\delta\}}}{D_{\bar\partial}(z,w)}.
\label{rough-cauchy-truncation}
\end{equation}
At fixed $\delta$ all kernels are bounded on the fixed compact supports.  Since
$q,\widetilde q,Q\in L^p_c\subset L^1_c$ and
$A,\widetilde A\in L^r_{\rm loc}\subset L^1_{\rm loc}$, the finite iterated
integrals are absolutely convergent, and repeated Fubini gives the recursive
branch expansions and the corresponding one-sided and mixed formulas.

The phase identities \eqref{exact-left-phase-split},
\eqref{exact-right-phase-split}, and
\eqref{exact-mixed-phase-split} are algebraic and are unchanged by the
truncation.  In the left one-sided term, for example, the center integration is
therefore
\begin{align*}
\frac{\tau}{\pi}\int
 e^{i\tau\psi_{z_0}(u_j)}\varphi(z_0)
 \bigl(A(s)-A(z_0)\bigr)\dd^2z_0
 =A(s)E_\tau\varphi(u_j)-E_\tau(\varphi A)(u_j).
\end{align*}
This proves the truncated version of
\eqref{rough-left-born-output-identity}.  The right identity follows by the changes of orientation and phase sign in Remark~\ref{rem-cauchy-involution}.  In
the mixed term, expand the two terminal differences before integrating in
$z_0$; the four resulting center integrals are exactly the four terms in
\eqref{rough-mixed-born-bracket}.  The decompositions
\eqref{rough-left-born-bracket-decomp},
\eqref{rough-right-born-bracket-decomp}, and
\eqref{rough-mixed-born-bracket-decomp} are pointwise algebraic identities.

We now let $\delta\downarrow0$.  The absolute values of all truncated branch
kernels are bounded by the corresponding untruncated positive kernels.  For the
pre-averaged formulas, Lemma~\ref{lem-fixed-order-absolute-convergence} supplies
an integrable majorant in all branch, root, and center variables.  Dominated
convergence therefore gives the untruncated branch expansions and permits the
center integration to be performed first.  Since the full expanded kernels are
absolutely integrable, repeated Fubini also identifies these multiple integrals
with the recursively defined nontruncated Cauchy operators.

The displayed terms are separately integrable by the one-sided $L^1$--$L^2$
bounds and the four mixed center-density bounds.  For the product error, the input
$\varphi A\widetilde A$ is compactly supported and belongs to $L^{r/2}\subset L^1$.  Hence \eqref{Etau-L1-Linfty-fixed-tau} gives
$E_\tau(\varphi A\widetilde A)\in L^\infty$ for fixed $\tau$, and therefore
\begin{equation}
 H_{A\widetilde A}^\tau\big|_{Z_{j,k}}
 \in L^{r/2}(Z_{j,k})
\end{equation}
on the fixed center support $Z_{j,k}$.  It pairs there with
$I_{j,k}^{1,1}\in L^{(r/2)'}$.  The same bounds dominate the limit
$\delta\downarrow0$ term by term.
\end{proof}

\section{Vanishing of the fixed Born terms}\label{sec:fixed-cancellations}

We now apply the identities of Section~\ref{sec:positive-densities}. The density bounds make the principal amplitudes integrable and control the averaging errors. The mixed terms leave one product error, which is treated by duality in Sections~\ref{subsec-product-error} and~\ref{subsec-product-duality}.

\subsection{One-sided cancellations}

\begin{prop}\label{prop-fixed-one-sided-cancellations}
Let $1<p<2$.  For every fixed $j\geq0$ and every
$\varphi\in C_0^\infty(\Omega)$,
\begin{equation}\label{fixed-linear-cancellation}
        \mathcal B_{j,\tau}=o(1)
        \qquad (\tau\to+\infty).
\end{equation}
For every fixed $k\geq0$ one likewise has
\begin{equation}\label{fixed-right-linear-cancellation}
        \widetilde{\mathcal B}_{k,\tau}=o(1).
\end{equation}
\end{prop}

\begin{proof}
We prove the assertion for the left branch; the right branch is symmetric. Proposition~\ref{prop-direct-rough-fixed-order} gives \eqref{fixed-left-signed-output-integral}, and the bracket is decomposed in \eqref{fixed-left-signed-output-bracket}. All changes in the order of integration are justified there at the original $L^p$ regularity.

For the principal term, the contribution containing $A(s)$ satisfies
\begin{align}
&\int |Q(x)|\,|\mathcal K_j^L(x;\lambda,\eta,s)|
 |\varphi(u_j)|\,|A(s)|\dd^2x\dd\Xi_j^L \notag\\
&\qquad=\int |\varphi(u)|J_j^{|A|}(u)\dd^2u
 \le \|\varphi\|_{L^\infty}\|J_j^{|A|}\|_{L^1},
\label{fixed-left-principal-terminal-bound}
\end{align}
while the contribution containing $A(u_j)$ satisfies
\begin{align}
&\int |Q(x)|\,|\mathcal K_j^L(x;\lambda,\eta,s)|
 |\varphi(u_j)A(u_j)|\dd^2x\dd\Xi_j^L \notag\\
&\qquad=\int |\varphi(u)A(u)|J_j^1(u)\dd^2u
 \le \|\varphi A\|_{L^2}\|J_j^1\|_{L^2}.
\label{fixed-left-principal-output-bound}
\end{align}
Both bounds are finite by Lemma~\ref{lem-one-sided-output-densities}.  Hence the principal term in
\eqref{fixed-left-signed-output-bracket} has an $L^1$ amplitude in all branch
variables.  If $j\geq1$, the residual phase $P_j^L$ depends only on
$(\lambda,\eta,s)$ and is nondegenerate on every fiber $u_j=\mathrm{const}$ by
\eqref{exact-left-phase-split}--\eqref{exact-left-residual-phase}.  We apply
Lemma~\ref{lem-quadratic-RL} with passive variables $(x,u_j)$, the root $x$
being absent from the phase, and with $w$ coordinates on the fiber of
$(\lambda,\eta,s)$, of real dimension $4j\ge4$.  Hence this term tends to zero.  If
$j=0$, then $u_0=s$ and the principal difference vanishes identically.

For the term containing $H_1^\tau$, the push-forward identity and \eqref{one-sided-left-weighted-L1} give
\begin{align}
&\left|\int e^{i\tau P_j^L}Q(x)\mathcal K_j^L
 A(s)H_1^\tau(u_j)\dd^2x\dd\Xi_j^L\right| \notag\\
&\qquad\le \|H_1^\tau\|_{L^\infty}
 \|J_j^{|A|}\|_{L^1}=o(1),
\label{fixed-left-smooth-error}
\end{align}
because $H_1^\tau=E_\tau\varphi-\varphi$ tends uniformly to zero by Lemma~\ref{lem-Etau}.  Similarly,
\begin{align}
&\left|\int e^{i\tau P_j^L}Q(x)\mathcal K_j^L
 H_A^\tau(u_j)\dd^2x\dd\Xi_j^L\right| \notag\\
&\qquad\le \|H_A^\tau\|_{L^2}\|J_j^1\|_{L^2}=o(1),
\label{fixed-left-L2-error}
\end{align}
since $\varphi A\in L^2$ on its compact support and Lemma
\ref{lem-Etau} gives $H_A^\tau\to0$ strongly in $L^2$.  Combining
\eqref{fixed-left-principal-terminal-bound}--\eqref{fixed-left-L2-error}
proves \eqref{fixed-linear-cancellation}.

The right-branch assertion follows from \eqref{rough-right-born-output-identity}, \eqref{rough-right-born-bracket-decomp}, and the corresponding estimates in Lemma~\ref{lem-one-sided-output-densities}.
\end{proof}

\subsection{The mixed oscillatory center kernel}

For the mixed term, we also retain the phase when integrating out the variables other than the signed center. Denote the resulting oscillatory density by $\mathcal G_{\tau,j,k}$. In contrast with the positive density $I_{j,k}^{1,1}$, it contains $e^{i\tau P_{j,k}}$, but its absolute value is bounded by that positive density.

For fixed $j,k$, let
\begin{equation}\label{rough-center-hat-variables}
 \widehat\Xi_{j,k}
 =(x,\lambda,\eta,s,\mu,\nu),
 \qquad
 \dd\widehat\Xi_{j,k}
 =\dd^2x\,\dd\Xi_j^L
   \prod_{m=1}^k\dd^2\mu_m\dd^2\nu_m.
\end{equation}
Thus $\widehat\Xi_{j,k}$ contains every mixed integration variable except
for the right terminal variable $t$.  For $c\in\C$, set
\begin{equation}\label{rough-center-terminal-substitution}
 t_c(\widehat\Xi_{j,k})
 =c+x-u_j-\sum_{m=1}^k(\mu_m-\nu_m),
\end{equation}
and let $\Xi_c$ denote the complete mixed variable list obtained by inserting
$t=t_c(\widehat\Xi_{j,k})$.  The empty sum is understood as zero.  Since
\[
 c_{j,k}
 =-x+u_j+\sum_{m=1}^k(\mu_m-\nu_m)+t,
\]
the affine change of variables
\begin{equation}\label{rough-center-change-variables}
 (t,\widehat\Xi_{j,k})
 \longmapsto(c_{j,k},\widehat\Xi_{j,k})
\end{equation}
has real Jacobian determinant one.

\begin{lem}
\label{lem-rough-oscillatory-center-kernel}
Let $1<p<2$, fix $j,k\geq0$, and put
\begin{equation}\label{product-duality-exponents}
 \frac r2=\frac{p}{2-p}>1,
 \qquad
 \left(\frac r2\right)'=\frac{r}{r-2}
     =\frac{p}{2(p-1)}.
\end{equation}
For $\tau\geq2$, define for almost every $c\in\C$
\begin{equation}\label{rough-oscillatory-center-kernel-def}
 \mathcal G_{\tau,j,k}(c)
 =\int
   e^{i\tau P_{j,k}(\Xi_c)}
   \mathcal K_{j,k}^{\rm mix}(\Xi_c)
   \dd\widehat\Xi_{j,k}.
\end{equation}
Then there is a fixed compact set $Z_{j,k}\Subset\C$, depending only on
the orders and the fixed supports and cutoffs, such that the integral in
\eqref{rough-oscillatory-center-kernel-def} is absolutely convergent for
almost every $c$, the resulting function is measurable, and
$\operatorname{supp}\mathcal G_{\tau,j,k}\subset Z_{j,k}$.  Moreover,
\begin{equation}\label{rough-center-kernel-pointwise-bound}
 |\mathcal G_{\tau,j,k}(c)|\leq I_{j,k}^{1,1}(c)
 \qquad\text{for a.e. }c.
\end{equation}
Consequently,
\begin{equation}\label{rough-center-kernel-Lmprime}
 \|\mathcal G_{\tau,j,k}\|_{L^{(r/2)'}}
 \leq C_{j,k,p}
 \|Q\|_{L^p}\|q\|_{L^p}^j
 \|\widetilde q\|_{L^p}^k,
\end{equation}
with a constant independent of $\tau$.

Moreover, for every $H\in L^{r/2}_{\rm loc}(\C)$,
\begin{align}
 \int H(c)\mathcal G_{\tau,j,k}(c)\dd^2c
 &=\int e^{i\tau P_{j,k}(\Xi)}
   \mathcal K_{j,k}^{\rm mix}(\Xi)
   H(c_{j,k})
   \dd^2x\dd\Xi_j^L\dd\Xi_k^R.
 \label{rough-center-kernel-pairing}
\end{align}
Both sides are absolutely convergent.
\end{lem}

\begin{proof}
Apply the determinant-one substitution
\eqref{rough-center-change-variables} to the positive mixed measure.  The
exact push-forward identity \eqref{exact-mixed-center-pushforward}, with
$F=G=1$, gives for almost every $c$
\begin{equation}\label{rough-center-absolute-fiber-identity}
 \int
   |\mathcal K_{j,k}^{\rm mix}(\Xi_c)|
   \dd\widehat\Xi_{j,k}
 =I_{j,k}^{1,1}(c).
\end{equation}
The right-hand side belongs to $L^1$ and is measurable by Lemma
\ref{lem-mixed-center-densities}.  Fubini--Tonelli therefore shows that the
integral defining $\mathcal G_{\tau,j,k}$ is absolutely convergent for almost every
$c$ and defines a measurable function.  Since the residual phase is real,
\eqref{rough-center-kernel-pointwise-bound} follows from
\eqref{rough-center-absolute-fiber-identity}.  The support assertion follows
from the corresponding assertion for $I_{j,k}^{1,1}$.

The exponent $(r/2)'$ in \eqref{product-duality-exponents} is finite because
$p>1$.  Taking $\kappa=(r/2)'$ in \eqref{mixed-center-unweighted-Lt} and using
\eqref{rough-center-kernel-pointwise-bound} proves
\eqref{rough-center-kernel-Lmprime}.

If $H\in L^{r/2}_{\rm loc}$, then H\"older's inequality,
the fixed support $Z_{j,k}$, and
\eqref{rough-center-absolute-fiber-identity} give
\[
 \int |H(c)|I_{j,k}^{1,1}(c)\dd^2c
 \leq \|H\|_{L^{r/2}(Z_{j,k})}
       \|I_{j,k}^{1,1}\|_{L^{(r/2)'}}<\infty.
\]
Thus Fubini's theorem and the inverse of
\eqref{rough-center-change-variables} are legitimate, and they give
\eqref{rough-center-kernel-pairing}.  The same majorant proves absolute
convergence of both sides.
\end{proof}

\subsection{Mixed cancellation and the remaining error}\label{subsec-mixed-reduction}

The density bounds control the principal amplitude and the three errors involving $\varphi$, $\varphi A$, or $\varphi\widetilde A$ in the center input. We isolate the product error in the following proof. Its duality estimate, Proposition~\ref{prop-direct-product-error-duality}, is proved in the remainder of this section.

\begin{prop}\label{prop-fixed-mixed-cancellations}
Let $1<p<2$.  For every fixed $j,k\geq0$ and every $\varphi\in C_0^\infty(\Omega)$,
\begin{equation}\label{fixed-mixed-cancellation}
        \mathcal M_{j,k,\tau}=o(1)
        \qquad (\tau\to+\infty).
\end{equation}
\end{prop}

\begin{proof}
By Proposition~\ref{prop-direct-rough-fixed-order}, the mixed functional is \eqref{fixed-mixed-signed-output-integral}, with the five-term bracket \eqref{rough-mixed-born-bracket-decomp}.

Expand the principal term according to the four possible placements of $A$ and $\widetilde A$. The contribution containing $A(s)\widetilde A(t)$ satisfies
\begin{align}
 &\int |\mathcal K_{j,k}^{\rm mix}(\Xi)|
 |\varphi(c_{j,k})|\,|A(s)|\,|\widetilde A(t)|
 \dd^2x\dd\Xi_j^L\dd\Xi_k^R
 \notag\\
 &\qquad\leq \|\varphi\|_{L^\infty}
 \|I_{j,k}^{|A|,|\widetilde A|}\|_{L^1}<\infty
 \label{fixed-mixed-terminal-terminal-L1}
\end{align}
by \eqref{mixed-center-two-weighted-L1}.  For the contribution containing $A(s)\widetilde A(c_{j,k})$,
the center push-forward and Cauchy--Schwarz give
\begin{align}
 &\int |\mathcal K_{j,k}^{\rm mix}(\Xi)|
 |\varphi(c_{j,k})|\,|A(s)|\,|\widetilde A(c_{j,k})|
 \dd^2x\dd\Xi_j^L\dd\Xi_k^R
 \notag\\
 &\qquad=
 \int |\varphi(c)\widetilde A(c)|I_{j,k}^{|A|,1}(c)\dd^2c
 \notag\\
 &\qquad\leq
 \|\varphi\widetilde A\|_{L^2}
 \|I_{j,k}^{|A|,1}\|_{L^2}<\infty.
 \label{fixed-mixed-left-terminal-right-center-L1}
\end{align}
Here we used \eqref{mixed-center-left-weighted-L2}; the symmetric term is controlled by \eqref{mixed-center-right-weighted-L2}.  Pair the contribution containing $A(c_{j,k})\widetilde A(c_{j,k})$ in the conjugate exponents
$r/2$ and $(r/2)'$.  It satisfies
\begin{align}
 &\int |\mathcal K_{j,k}^{\rm mix}(\Xi)|
 |\varphi(c_{j,k})A(c_{j,k})\widetilde A(c_{j,k})|
 \dd^2x\dd\Xi_j^L\dd\Xi_k^R
 \notag\\
 &\qquad=
 \int |\varphi(c)A(c)\widetilde A(c)|I_{j,k}^{1,1}(c)\dd^2c
 \notag\\
 &\qquad\leq
 \|\varphi A\widetilde A\|_{L^{r/2}}
 \|I_{j,k}^{1,1}\|_{L^{(r/2)'}}<\infty
 \label{fixed-mixed-center-center-L1}
\end{align}
by \eqref{mixed-center-unweighted-Lt}.  Therefore the principal term in
\eqref{rough-mixed-born-bracket-decomp} is an $L^1$ amplitude in all fixed-order
variables.  By \eqref{exact-mixed-phase-split} and
Lemma~\ref{lem-signed-quadratic-normal-form}, the residual phase
$P_{j,k}$ is nondegenerate on each fiber $c_{j,k}=\mathrm{const}$.  Every
mixed variable, including the root $x$, occurs in the signed list, so we apply
Lemma~\ref{lem-quadratic-RL} with the single passive variable $c_{j,k}$.  The
fiber has real dimension $4(j+k+1)\ge4$; already when $j=k=0$ the constraint
$c=-x+s+t$ leaves two free complex variables.  Consequently
\begin{equation}\label{fixed-mixed-principal-decay}
 \int e^{i\tau P_{j,k}}\mathcal K_{j,k}^{\rm mix}
 \varphi(c_{j,k})(A(s)-A(c_{j,k}))
 (\widetilde A(t)-\widetilde A(c_{j,k}))
 \longrightarrow0.
\end{equation}

For the error involving the test function,
\begin{align}
 &\left|\int e^{i\tau P_{j,k}}\mathcal K_{j,k}^{\rm mix}
 A(s)\widetilde A(t)H_1^\tau(c_{j,k})\right|
 \notag\\
 &\qquad\leq
 \|H_1^\tau\|_{L^\infty}
 \|I_{j,k}^{|A|,|\widetilde A|}\|_{L^1}=o(1),
 \label{fixed-mixed-smooth-test-error}
\end{align}
because $H_1^\tau=E_\tau\varphi-\varphi$ tends to zero uniformly by
Lemma~\ref{lem-Etau}.

For the error involving $H_{\widetilde A}^\tau$, the center push-forward and \eqref{mixed-center-left-weighted-L2} yield
\begin{align}
 &\left|\int e^{i\tau P_{j,k}}\mathcal K_{j,k}^{\rm mix}
 A(s)H_{\widetilde A}^\tau(c_{j,k})\right|
 \notag\\
 &\qquad\leq
 \|I_{j,k}^{|A|,1}\|_{L^2}
 \|H_{\widetilde A}^\tau\|_{L^2}=o(1).
 \label{fixed-mixed-left-one-factor-error}
\end{align}
Indeed, $\varphi\widetilde A\in L^2_c$ and Lemma~\ref{lem-Etau} gives strong
$L^2$ convergence.  The symmetric estimate
\begin{equation}\label{fixed-mixed-right-one-factor-error}
 \left|\int e^{i\tau P_{j,k}}\mathcal K_{j,k}^{\rm mix}
 \widetilde A(t)H_A^\tau(c_{j,k})\right|
 \leq
 \|I_{j,k}^{1,|\widetilde A|}\|_{L^2}
 \|H_A^\tau\|_{L^2}=o(1)
\end{equation}
follows from \eqref{mixed-center-right-weighted-L2}.

For the product error, the pairing identity in Lemma~\ref{lem-rough-oscillatory-center-kernel} gives
\[
        \int H_{A\widetilde A}^\tau(c)\mathcal G_{\tau,j,k}(c)\dd^2c.
\]
The remaining assertion is therefore Proposition~\ref{prop-direct-product-error-duality}, applied with $F=\varphi A\widetilde A\in L^{r/2}_c$. That proposition, proved below, completes \eqref{fixed-mixed-cancellation}.
\end{proof}

\subsection{Transposing the product error}\label{subsec-product-error}

Fix the orders $j,k$ and write $F=\varphi A\widetilde A$. The only remaining pairing is
\[
 \int (E_\tau F-F)(c)\mathcal G_{\tau,j,k}(c)\,\dd^2c.
\]
The $L^2$ convergence of $E_\tau$ need not apply to this input. Instead, transpose the average in the bilinear pairing. The phase identity \eqref{exact-mixed-phase-split} then restores the original mixed phase:
\[
 \psi_{z_0}(c_{j,k})+P_{j,k}
 =-\psi_{z_0}(x)+\Phi_j^L(z_0)+\Phi_k^R(z_0).
\]
Both factors $A$ and $\widetilde A$ are already in $F$, so neither remains at a branch terminal. Put
\begin{equation}\label{direct-transpose-branch-def}
 \Pi_j^{z_0,\tau}=\mathfrak L_j^{z_0,\tau}[1],
 \qquad \widetilde\Pi_k^{z_0,\tau}=\mathfrak R_k^{z_0,\tau}[1].
\end{equation}
The input $1$ leaves the smooth cutoff $\chi$ in each terminal Cauchy transform. Lemma~\ref{lem-direct-rough-center-transpose} will prove the exact identity
\begin{equation}\label{direct-rough-center-transpose-identity}
 (E_\tau^t\mathcal G_{\tau,j,k})(z_0)
 =\frac{\tau}{\pi}\int_{\C}
 Q(x)e^{-i\tau\psi_{z_0}(x)}
 \Pi_j^{z_0,\tau}(x)\widetilde\Pi_k^{z_0,\tau}(x)\,\dd^2x.
\end{equation}
The estimate proceeds from each smooth terminal cutoff toward the root: the terminal Cauchy transform gains $\tau^{-1/2}$, the first two-Cauchy block places a positive-order branch in $L^\infty$, and later blocks preserve this bound. The two gains offset the prefactor $\tau$.

\begin{lem}\label{lem-reverse-branch-terminal-gain}
Let $1<p<2$, choose
\begin{equation}\label{reverse-branch-s0-choice}
        b_0>2p',
\end{equation}
and fix compact sets of centers and spatial variables
$Z\Subset\C$ and $Y\Subset\C$.  Assume that $Y$ contains the supports of a
fixed function $f\in C_0^\infty(\C)$, of all coefficient functions below, and
of all cutoffs occurring in the recursion.  Let
$F_1,\ldots,F_n\in L^p_c(\C)$, put
\begin{equation}\label{reverse-branch-terminal-leaf}
        \Pi_0=\mathcal C_{\tau,z_0}^{\varepsilon_0,\mathfrak d_0}f,
\end{equation}
where $\varepsilon_0\in\{\pm1\}$ and
$\mathfrak d_0\in\{\partial,\bar\partial\}$, and for $\ell\geq0$ define
\begin{equation}\label{transposed-branch-recursion}
\Pi_{\ell+1}
=\mathcal C_{\tau,z_0}^{\varepsilon_{\ell,1},\mathfrak d_{\ell,1}}
 \left(
  \chi_{\ell,1}\,
  \mathcal C_{\tau,z_0}^{\varepsilon_{\ell,2},\mathfrak d_{\ell,2}}
  \bigl(\chi_{\ell,2}F_{\ell+1}\Pi_\ell\bigr)
 \right).
\end{equation}
Here the signs and Cauchy orientations are arbitrary and the cutoffs are fixed.
Then, uniformly for $z_0\in Z$ and $\tau\geq2$,
\begin{equation}\label{transposed-branch-terminal-bound}
        \|\Pi_0\|_{L^{b_0}(Y)}
        \leq C\tau^{-1/2}\|f\|_{W^{1,a_{b_0}}},
        \qquad
        \frac1{a_{b_0}}=\frac12+\frac1{b_0},
\end{equation}
and, for every $1\leq\ell\leq n$,
\begin{equation}\label{transposed-branch-Linf-bound}
        \|\Pi_\ell\|_{L^\infty(Y)}
        \leq C_\ell\tau^{-1/2}\|f\|_{W^{1,a_{b_0}}}
        \prod_{h=1}^{\ell}\|F_h\|_{L^p}.
\end{equation}
The constants may depend on $Y,Z$, the order, and the fixed cutoffs, but not on
$z_0$ or $\tau$.
\end{lem}

\begin{proof}
The terminal estimate \eqref{transposed-branch-terminal-bound} follows from
Lemma~\ref{lem-terminal-cauchy-gain}, applied with both the target set and the
support sets contained in the fixed compact set $Y$.

To propagate this estimate through a two-Cauchy block, choose $R>0$, depending only
on $Y$ and the fixed cutoff supports, so large that every difference of two
variables occurring in a block lies in $B(0,R)$.  Taking absolute values in
\eqref{transposed-branch-recursion} gives, for $x\in Y$,
\begin{equation}\label{reverse-branch-local-majorant}
        |\Pi_{\ell+1}(x)|
        \leq C\,K_R*
        \bigl(\mathbf 1_Y|F_{\ell+1}\Pi_\ell|\bigr)(x),
        \qquad K_R=k_R*k_R.
\end{equation}

For the first block, define $a$ by
\begin{equation}\label{reverse-branch-first-exponent}
        \frac1a=\frac1p+\frac1{b_0}.
\end{equation}
The choice $b_0>2p'$ implies in particular $b_0>p'$, and hence
$1/a<1$, so $a>1$ and $a'<\infty$.  H\"older's inequality,
\eqref{transposed-branch-terminal-bound}, and $K_R\in L^{a'}$ yield
\begin{align*}
 \|\Pi_1\|_{L^\infty(Y)}
 &\leq C\|K_R\|_{L^{a'}}
       \|F_1\Pi_0\|_{L^a(Y)}\\
 &\leq C\tau^{-1/2}\|f\|_{W^{1,a_{b_0}}}\|F_1\|_{L^p}.
\end{align*}
For every later block, \eqref{reverse-branch-local-majorant},
$K_R\in L^{p'}$, and the induction hypothesis give
\begin{align*}
 \|\Pi_{\ell+1}\|_{L^\infty(Y)}
 &\leq C\|K_R\|_{L^{p'}}
       \|F_{\ell+1}\Pi_\ell\|_{L^p(Y)}\\
 &\leq C\|F_{\ell+1}\|_{L^p}
       \|\Pi_\ell\|_{L^\infty(Y)}.
\end{align*}
Iterating this inequality proves \eqref{transposed-branch-Linf-bound}.
\end{proof}

\begin{cor}
\label{cor-reverse-born-branch-gains}
Fix the Born orders $j,k\geq0$ and one compact set $Y$ containing
$\operatorname{supp}Q$, the supports of $q$ and $\widetilde q$, and every
cutoff support in the two branches. The left branch $\Pi_j=\mathfrak L_j[1]$, built from $j$
blocks with coefficient $q$, satisfies, uniformly on the center compact set
$Z$,
\begin{align}
 \|\Pi_0^{z_0,\tau}\|_{L^{b_0}(Y)}
 &\leq C\tau^{-1/2},
 \label{reverse-left-zero-gain}\\
 \|\Pi_j^{z_0,\tau}\|_{L^\infty(Y)}
 &\leq C_j\tau^{-1/2}\|q\|_{L^p}^j,
 \qquad j\geq1.
 \label{reverse-left-positive-gain}
\end{align}
The right branch with terminal input $1$ satisfies
\begin{align}
 \|\widetilde\Pi_0^{z_0,\tau}\|_{L^{b_0}(Y)}
 &\leq C\tau^{-1/2},
 \label{reverse-right-zero-gain}\\
 \|\widetilde\Pi_k^{z_0,\tau}\|_{L^\infty(Y)}
 &\leq C_k\tau^{-1/2}\|\widetilde q\|_{L^p}^k,
 \qquad k\geq1.
 \label{reverse-right-positive-gain}
\end{align}
The signs, Cauchy orientations, and fixed terminal cutoffs may be arbitrary.
\end{cor}

\begin{proof}
Apply Lemma~\ref{lem-reverse-branch-terminal-gain} with every $F_h=q$ on the left and
every $F_h=\widetilde q$ on the right.  The lemma is uniform in the signs and
Cauchy orientations, so it applies to the recursions defining
$\mathfrak L_j[1]$ and $\mathfrak R_k[1]$.
\end{proof}

\paragraph{The order $j=k=0$.}
At the lowest mixed order the signed outputs are simply $s$ and $t$, so that $c=-x+s+t$.  Truncating the Cauchy kernels and applying Fubini, as in the general argument below, gives
\begin{equation}\label{model-zero-transpose}
 (E_\tau^t\mathcal G_{\tau,0,0})(z_0)
 =\frac{\tau}{\pi}\int Q(x)e^{-i\tau\psi_{z_0}(x)}
 \Pi_0^{z_0,\tau}(x)\widetilde\Pi_0^{z_0,\tau}(x)\,\dd^2x.
\end{equation}
Both order-zero branches are terminal oscillatory Cauchy transforms of the fixed cutoff $\chi$.  If $b_0>2p'$, Corollary~\ref{cor-reverse-born-branch-gains} gives
\[
 \|\Pi_0^{z_0,\tau}\|_{L^{b_0}}
 +\|\widetilde\Pi_0^{z_0,\tau}\|_{L^{b_0}}
 \leq C\tau^{-1/2}.
\]
Since $b_0/2>p'$ and the spatial variables lie in a fixed compact set, H\"older's inequality bounds the integral in \eqref{model-zero-transpose} by $C\tau^{-1}$, cancelling the prefactor $\tau$.  At positive order, the first two-Cauchy block places the corresponding branch in $L^\infty$ and the same estimate applies.

We may now transpose the center kernel without regularizing the coefficients.

\begin{lem}
\label{lem-direct-rough-center-transpose}
Let $1<p<2$, fix $j,k\geq0$, and let $Z\Subset\C$. The exact identity \eqref{direct-rough-center-transpose-identity} holds for almost every $z_0\in Z$. Moreover,
\begin{equation}\label{direct-rough-center-transpose-Linf}
 \|E_\tau^t\mathcal G_{\tau,j,k}\|_{L^\infty(Z)}
 \leq C_{j,k,Z}
 \|Q\|_{L^p}\|q\|_{L^p}^j
 \|\widetilde q\|_{L^p}^k,
\end{equation}
uniformly for $\tau\geq2$.
\end{lem}

\begin{proof}
Since $\mathcal G_{\tau,j,k}$ has fixed compact support and belongs uniformly to
$L^{(r/2)'}$, it belongs to $L^1$.  Hence $E_\tau^t\mathcal G_{\tau,j,k}$ is defined by
\eqref{Etau-transpose-def}.  Substituting
\eqref{rough-oscillatory-center-kernel-def}, using
\eqref{rough-center-absolute-fiber-identity}, and applying Fubini gives
\begin{align*}
 (E_\tau^t\mathcal G_{\tau,j,k})(z_0)
 &=\frac{\tau}{\pi}
 \int e^{i\tau\{\psi_{z_0}(c_{j,k})+P_{j,k}(\Xi)\}}
 \mathcal K_{j,k}^{\rm mix}(\Xi)
 \dd^2x\dd\Xi_j^L\dd\Xi_k^R.
\end{align*}
The exchange is legitimate for each fixed $\tau$, since its absolute majorant
is
\[
 \frac{\tau}{\pi}\int I_{j,k}^{1,1}(c)\dd^2c<\infty.
\]
By \eqref{exact-mixed-phase-split},
\[
 \psi_{z_0}(c_{j,k})+P_{j,k}(\Xi)
 =-\psi_{z_0}(x)+\Phi_j^L(z_0)+\Phi_k^R(z_0).
\]
The branch-mass bounds in Lemma~\ref{lem-branch-output-bounds} show that
\eqref{exact-left-expanded} and \eqref{exact-right-expanded} with terminal
factor $1$ are absolutely convergent for the original $L^p$ coefficients; equivalently,
one may repeat the truncation argument of
Proposition~\ref{prop-direct-rough-fixed-order}.  The exact factorization
\eqref{exact-mixed-kernel}, followed by Fubini in the two branch variables,
therefore yields
\begin{align*}
 (E_\tau^t\mathcal G_{\tau,j,k})(z_0)
 &=\frac{\tau}{\pi}\int Q(x)e^{-i\tau\psi_{z_0}(x)}
 \left(\int e^{i\tau\Phi_j^L(z_0)}
       \mathcal K_j^L\dd\Xi_j^L\right)\notag\\
 &\hspace{42mm}\times
 \left(\int e^{i\tau\Phi_k^R(z_0)}
       \mathcal K_k^R\dd\Xi_k^R\right)\dd^2x,
\end{align*}
which is \eqref{direct-rough-center-transpose-identity} by
\eqref{exact-left-expanded}, \eqref{exact-right-expanded}, and
\eqref{direct-transpose-branch-def}.

Choose the compact set $Y$ and the exponent $b_0>2p'$ as in
Corollary~\ref{cor-reverse-born-branch-gains}, with
$\operatorname{supp}Q\subset Y$.  The two factors in
\eqref{direct-rough-center-transpose-identity} each contribute a gain
$\tau^{-1/2}$.  We distinguish the possible branch orders.

If $j=k=0$, choose $\gamma\in(1,\infty)$ from
\[
 \frac1p+\frac2{b_0}+\frac1\gamma=1;
\]
the choice $b_0>2p'$ makes $\gamma$ well defined.  H\"older's inequality on $Y$
and \eqref{reverse-left-zero-gain}--\eqref{reverse-right-zero-gain} give
\[
 \int_Y|Q \Pi_0^{z_0,\tau}\widetilde\Pi_0^{z_0,\tau}|
 \leq C_Y\tau^{-1}\|Q\|_{L^p}.
\]
If $j\geq1$ and $k=0$, choose $\gamma$ from
$1/p+1/b_0+1/\gamma=1$ and use
\eqref{reverse-left-positive-gain} and \eqref{reverse-right-zero-gain}:
\[
 \int_Y|Q \Pi_j^{z_0,\tau}\widetilde\Pi_0^{z_0,\tau}|
 \leq C_{j,Y}\tau^{-1}\|Q\|_{L^p}\|q\|_{L^p}^j.
\]
The case $j=0$, $k\geq1$ is symmetric.  Finally, if $j,k\geq1$, then
$\|Q\|_{L^1(Y)}\leq C_Y\|Q\|_{L^p}$ and
\eqref{reverse-left-positive-gain}--\eqref{reverse-right-positive-gain} give
\[
 \int_Y|Q \Pi_j^{z_0,\tau}\widetilde\Pi_k^{z_0,\tau}|
 \leq C_{j,k,Y}\tau^{-1}
 \|Q\|_{L^p}\|q\|_{L^p}^j\|\widetilde q\|_{L^p}^k.
\]
The prefactor $\tau/\pi$ in
\eqref{direct-rough-center-transpose-identity} cancels the gain $\tau^{-1}$
in all four cases.  This proves
\eqref{direct-rough-center-transpose-Linf}.
\end{proof}

\subsection{Duality for the product error}\label{subsec-product-duality}

The uniform transpose estimate, together with the positive majorant for $\mathcal G_{\tau,j,k}$, allows approximation of the center input in $L^{r/2}$. Only the center input $F$ is approximated. The coefficients defining $\mathcal G_{\tau,j,k}$ are unchanged.

\begin{prop}
\label{prop-direct-product-error-duality}
Let $1<p<2$, fix $j,k\geq0$.  For every compactly supported $F\in L^{r/2}(\C)$,
\begin{equation}\label{direct-product-error-duality-conclusion}
 \int_{\C}(E_\tau F-F)(c)\mathcal G_{\tau,j,k}(c)\dd^2c
 \longrightarrow0
 \qquad (\tau\to\infty).
\end{equation}
The functionals on the left are uniformly bounded on every fixed
compactly supported subspace of $L^{r/2}$.
\end{prop}

\begin{proof}
Choose one compact set $Z$ containing $\operatorname{supp}F$ and, later,
the supports of its smooth approximants.  Since $F\in L^{r/2}_c\subset L^1$,
Fubini and the definition of the bilinear transpose
give
\begin{equation}\label{direct-product-error-transpose-pairing}
 \int (E_\tau F)(c)\mathcal G_{\tau,j,k}(c)\dd^2c
 =\int F(z_0)(E_\tau^t\mathcal G_{\tau,j,k})(z_0)\dd^2z_0.
\end{equation}
Using Lemma~\ref{lem-direct-rough-center-transpose},
Lemma~\ref{lem-rough-oscillatory-center-kernel}, and the finite measure of
$Z$, we obtain
\begin{align}
 \left|\int(E_\tau F-F)\mathcal G_{\tau,j,k}\right|
 &\leq \|F\|_{L^1(Z)}
       \|E_\tau^t\mathcal G_{\tau,j,k}\|_{L^\infty(Z)}
       +\|F\|_{L^{r/2}}\|\mathcal G_{\tau,j,k}\|_{L^{(r/2)'}}\notag\\
 &\leq C_{j,k,Z}
 \|Q\|_{L^p}\|q\|_{L^p}^j\|\widetilde q\|_{L^p}^k
 \|F\|_{L^{r/2}}.
 \label{direct-product-error-functional-bound}
\end{align}
If $F\in C_0^\infty(\C)$, Lemma~\ref{lem-Etau} gives
$\|E_\tau F-F\|_{L^\infty}\to0$.  The compact support of
$\mathcal G_{\tau,j,k}$ and \eqref{rough-center-kernel-Lmprime} imply the uniform bound
\[
 \|\mathcal G_{\tau,j,k}\|_{L^1}
 \leq |Z_{j,k}|^{2/r}\|\mathcal G_{\tau,j,k}\|_{L^{(r/2)'}}\leq C_{j,k},
\]
so \eqref{direct-product-error-duality-conclusion} follows for smooth $F$.
For general $F\in L^{r/2}_c$, choose $F_n\in C_0^\infty$ with common compact support and $F_n\to F$ in $L^{r/2}$. Apply \eqref{direct-product-error-functional-bound} to $F-F_n$ and the smooth result to $F_n$. First let $\tau\to\infty$ with $n$ fixed, and then let $n\to\infty$, using the uniform bound for $F-F_n$. This proves \eqref{direct-product-error-duality-conclusion}.
\end{proof}

Taking $F=\varphi A\widetilde A$ completes the proof of Proposition~\ref{prop-fixed-mixed-cancellations}.

\section{Completion of the proof}\label{sec:final-proof}\label{subsec-completion-remainder}

\begin{prop}\label{prop-Lp-remainders}
Let $1<p<2$.  For every $\varphi\in C_0^\infty(\Omega)$,
\begin{equation}\label{Lp-remainder-estimate}
\frac{\tau}{\pi}\int\varphi(z_0)
\left[
\int Qe^{i\tau\psi_{z_0}}
\bigl((1+\zeta_\tau)(1+\widetilde\zeta_\tau)-1\bigr)\dd^2z
\right]\dd^2z_0=o(1).
\end{equation}
Equivalently, $\mathcal R_\tau=o(1)$.
\end{prop}

\begin{proof}
Equation~\eqref{finite-tail-born-split} expresses $\mathcal R_\tau$ as three finite sums, with $j<N_1$, $k<N_1$, and $j+k<N_2$, plus an $o(1)$ contribution from the complementary tails.  The first two sums are $o(1)$ by Proposition~\ref{prop-fixed-one-sided-cancellations}, and the mixed sum is $o(1)$ by Proposition~\ref{prop-fixed-mixed-cancellations}.
\end{proof}

\begin{proof}[Proof of Theorem \ref{main-thm}]
Assume first that $1<p<2$ and put $Q=V-\widetilde V$.  Apply Lemma~\ref{lem-alessandrini} to the CGO solutions for $V$ and $\widetilde V$.  Equality of the weak Dirichlet-to-Neumann maps gives \eqref{alessandrini} for almost every $z_0$.  Lemma~\ref{lem-center-measurability} allows us to multiply this identity by $\tau\varphi(z_0)/\pi$, integrate over $z_0$, and expand the two Neumann series term by term.  Hence, for every $\varphi\in C_0^\infty(\Omega)$,
\begin{equation}\label{final-averaged-identity}
 0=\mathcal L_\tau+\mathcal R_\tau,
\end{equation}
where $\mathcal L_\tau$ is the leading functional in
\eqref{leading-functional-def} and $\mathcal R_\tau$ is the sum of all
one-sided and mixed Born terms in
\eqref{remainder-functional-as-born-sums}.

By \eqref{leading-limit},
\begin{equation}\label{final-leading-limit}
 \mathcal L_\tau=\int_\Omega Q\varphi\,\dd^2z+o(1).
\end{equation}
For the remainder, Proposition~\ref{prop-Lp-remainders} gives
\begin{equation}\label{final-remainder-limit}
 \mathcal R_\tau=o(1).
\end{equation}
Letting $\tau\to\infty$ in \eqref{final-averaged-identity} and using
\eqref{final-leading-limit}--\eqref{final-remainder-limit}, we obtain
\begin{equation*}
 \int_\Omega Q(z)\varphi(z)\,\dd^2z=0
 \qquad\text{for every }\varphi\in C_0^\infty(\Omega).
\end{equation*}
Thus $Q=0$ in $\mathcal D'(\Omega)$ and therefore almost everywhere in $\Omega$.

If $p\geq2$, choose any $p_0\in(1,2)$.  Since $\Omega$ is bounded, $V,\widetilde V\in L^{p_0}(\Omega)$, and the result just proved applies without changing the weak equations or the Dirichlet-to-Neumann maps.
\end{proof}

\section*{Acknowledgements}
C.C. was supported by the National Science and Technology Council (NSTC), Taiwan, under grant No.~113-2115-M-A49-018-MY3.

\section*{AI use disclosure statement} Generative AI tools have been used in the production of this paper for drafting, proof exploration, and quality control. The authors have reviewed the manuscript and assert both credit and responsibility for its contents.

\bibliographystyle{plain}
\bibliography{schrodinger-Lp-arxiv}

\end{document}